\documentclass[12pt, a4paper, reqno]{amsart}
\usepackage{amssymb}
\usepackage{amsthm}
\usepackage{tikz}
\usetikzlibrary{arrows,positioning} % for diagrams
\usepackage{mathrsfs}	
\usepackage{bbm}
\usepackage{ stmaryrd }

\usepackage{hyperref}
\hypersetup{
	colorlinks=true,
	linkcolor=blue,
	filecolor=magenta,  
	urlcolor=cyan,
}

\usepackage{hyperref}% http://ctan.org/pkg/hyperref
\usepackage{cleveref}% http://ctan.org/pkg/cleveref

\usepackage{amsfonts}
\usepackage{amsthm}
\usepackage{amsmath}
\usepackage{amscd}
\usepackage[latin2]{inputenc}
\usepackage{t1enc}
\usepackage[mathscr]{eucal}
\usepackage{indentfirst}
\usepackage{graphicx}
\usepackage{graphics}
\usepackage{epic}
\numberwithin{equation}{section}
\usepackage[margin=2.9cm]{geometry}
\usepackage{epstopdf} 

\usepackage{multirow}
\usepackage{makecell}

\def\End{\operatorname{End}}

\def\ker{\operatorname{ker}}

\def\im{\operatorname{Im}}

\def\dim{\operatorname{dim}}

\def\id{\operatorname{id}}

\def\Ind{\operatorname{Ind}}

\def\Ber{\operatorname{Ber}}
\def\Ind{\operatorname{Ind}}
\def\Ad{\operatorname{Ad}}
\def\colim{\operatorname{colim}}

\def\Hom{\operatorname{Hom}}

\def\C{\mathbb{C}}
\def\Q{\mathbb{Q}}
\def\R{\mathbb{R}}

\def\Z{\mathbb{Z}}

\def\OO{\mathcal{O}}

\def\UU{\mathcal{U}}

\def\EE{\mathcal{E}}

\def\b{\mathfrak{b}}

\def\n{\mathfrak{n}}
\def\p{\mathfrak{p}}
\def\q{\mathfrak{q}}
\def\g{\mathfrak{g}}

\def\h{\mathfrak{h}}
\def\r{\mathfrak{r}}
\def\k{\mathfrak{k}}
\def\l{\mathfrak{l}}
\def\s{\mathfrak{s}}
\def\o{\mathfrak{o}}

\def\B{\textbf}

\def\d{\partial}

\def\ol{\overline}

\def\tr{\text{tr}}

\def\Spec{\text{Spec}}

\def\sub{\subseteq}

\def\xto{\xrightarrow}

\newtheorem{thm}{Theorem}[section]
\newtheorem{cor}[thm]{Corollary}
\newtheorem{lemma}[thm]{Lemma}
\newtheorem{prop}[thm]{Proposition}

\theoremstyle{definition}
\newtheorem{definition}[thm]{Definition}
\theoremstyle{remark}
\newtheorem{remark}[thm]{Remark}
\newtheorem{example}[thm]{Example}

\numberwithin{equation}{section}

\usepackage{relsize}
\usepackage{fancyhdr}
\usepackage[all,cmtip]{xy}

\begin{document}
	
\title{Spherical Indecomposable Representations of Lie Superalgebras}

\author[Alexander Sherman]{Alexander Sherman}

\maketitle
\pagestyle{plain}

\begin{abstract} We present a classification of all spherical indecomposable representations of classical and exceptional Lie superalgebras.  We also include information about stabilizers, symmetric algebras, and Borels for which sphericity is achieved.  In one such computation, the symmetric algebra of the standard module of $\mathfrak{osp}(m|2n)$ is computed, which in particular gives the representation-theoretic structure of polynomials on the complex supersphere.
\end{abstract}

\section{Introduction}\label{sec_intro}

Let $\g$ be a Lie superalgebra from the list $\g\l(m|n),\o\s\p(m|2n),\p(n),\q(n)$ or exceptional basic simple.  We want to describe all spherical indecomposable representations of $\g$, that is all finite-dimensional indecomposable $\g$-modules $V$ such that there exists an even vector $v\in V_{\ol{0}}$ and a Borel subalgebra $\b$ of $\g$ such that $(\b+\C\cdot\id_V)\cdot v=V$ (see sections \ref{sec_defns} and \ref{sec_tools} for precise definitions).  We add in the action of scalars for greater generality.

The problem of finding all spherical irreducible representations of reductive Lie algebras was solved by V. Kac in \cite{kac1980some}.  There, Kac classifies all visible linear groups acting irreducibly on a vector space, that is actions where the nilvariety has finitely many orbits.  From this he deduces all reductive linear groups acting irreducibly on a vector space with finitely many orbits (i.e. prehomogeneous vector spaces, which were originally classified by Sato and Kimura in \cite{sato1977classification}), leading to the classification of spherical irreducibles.  Later on, A. Leahy in a result communicated by F. Knop in \cite{leahy1998classification}, as well as Benson and Ratcliff in \cite{benson1996classification}, provided a classification of all spherical representations of reductive Lie algebras.

In the classical world, spherical representations are exactly representations which have multiplicity free symmetric algebras.    In \cite{howe1995perspectives}, Howe emphasizes such representations as being pervasive in the study of invariant theory.  Multiplicity-free properties of exterior algebras, and even of supersymmetric algebras are natural in this regard, and have been studied by T. Pecher in \cite{pecher2011multiplicity} and \cite{pecher2012classification}.

Geometrically, spherical representations may be viewed as basic and highly tractable examples of spherical varieties, which are of interest in geometry, representation theory, and combinatorics.

In the super world, the study of symmetric superspaces is of interest.  The representation theory of the space of functions as well as the algebra of invariant differential operators has been studied in recent years (see e.g. \cite{sahi2016capelli}, \cite{alldridge2012harish},\cite{alldridge2015spherical},\cite{sergeev2017symmetric},\cite{sahi2018capelli}).  In \cite{sahi2018capelli}, the authors associate to each simple Jordan superalgebra $J$ a supersymmetric pair $(\g,\k)$ via the TKK construction.  The vector superspace $J$ inherits a $\g$-action making it a simple spherical $\g$-module, and the symmetric superspace $G/K$ is realized as an open $G$-orbit.  This motivates the study of the space $J$, which is a spherical irreducible representation, and the authors describe the representation of the algebra of functions as well as the invariant differential operators. 

In this work, we provide more examples of superspaces which can be thought of as analogues of multiplicity-free representations in the super setting.  In \cite{sherman2019sphericalsupervar}, we expand on the notion of spherical supervarieties and develop connections to representation theory.

\subsection{Structure of Paper}  We will complete the introduction with a statement of the main results of the classification.  In sections \ref{sec_defns} and \ref{sec_tools} we explain our notation and develop the tools we use for the classification. Section \hyperref[proof_procedure]{4} explains the method of proof. Note that appendix \hyperref[App_borels]{A} discusses notation and structure of Borel subalgebras for $\g\l(m|n)$ and $\o\s\p(m|2n)$.

The rest of the paper goes through the classification case by case, section by section, for each algebra, in the order: $\g\l(m|n)$, $\o\s\p(m|2n)$, the exceptional basic simple algebras, $\p(n)$, and finally $\q(n)$.  In each case, in addition to classifying all spherical indecomposable representations, we compute for which Borels the representation is spherical and the (isomorphism class of) the stabilizer subalgebra (in $\g$) of an even vector in the open orbit.  In section \hyperref[App_sym]{10} we explain the structure of the symmetric algebra on the dual representation of most cases.  This includes the computation of functions on the supersphere as a codimension-one subvariety of $\C^{m|2n}$. 

\

The technique of proof is rather simple-minded.  We use the known parameterizations of dominant highest weights for each Lie superalgebra, impose the immediate restrictions sphericity gives, and then check any remaining candidates.  Applying odd reflections becomes our most powerful tool in the case of basic algebras.

\subsection{Statement of Main Results}

We work over the complex numbers.  The main results are as follows.  

For a quasireductive Lie superalgebra $\g$ and representation $\rho:\g\to\g\l(V)$, we say the triple $(V,\g,\rho)$ is \emph{spherical} if there exists a Borel subalgebra $\b$ of $\g$ and a vector $v\in V_{\ol{0}}$ such that $(\rho(\b)+\C\id_{V})v=V$.  We classify such triples $(V,\g,\rho)$ up to equivalence, where we say $(V,\g,\rho)\sim(V',\g',\rho')$ iff $V\cong V'$ and $\rho(\g)+\C\id_{V}$ maps to $\rho'(\g')+\C\id_{V'}$ under the induced isomorphism $\g\l(V)\to\g\l(V')$.  See \cref{sec_tools} for more details.  Note we only consider cases when $V_{\ol{1}}\neq0$, as otherwise we are considering the action of a reductive group, and this case is already known.

As an example, consider the quasireductive Lie superalgebra $\g$ with $\g_{\ol{0}}=0$ and $\g_{\ol{1}}=\C\langle X\rangle$, a one-dimensional odd abelian algebra.  This algebra has a spherical representation acting on the super vector space $\C^{1|1}$, with $X$ acting as the matrix $\begin{bmatrix}
0 & 0\\1 & 0 \end{bmatrix}$.  We refer to this representation as $U^{1|1}$.  It is perhaps the simplest non-trivial spherical representation of a Lie superalgebra that is not purely even.

\

For a representation $V$ of a superalgebra $\g$, we may consider its algebra of functions as a supervariety, which will be $S^\bullet V^*$.  It is interesting to decompose this space as a $\g$-module, and in particular to determine when each symmetric power is completely reducible.  In our case, all symmetric algebras will have multiplicity-free socles by our assumption of sphericity (see \cref{cmult_free}- note that the result still holds for $\g=\q(n)$ because the representations we come across have $(1|1)$-dimensional highest weight space).

The following is the list, up to equivalence, of the infinite families of spherical indecomposable modules, all of which happen to be irreducible, and criteria for when their algebra of functions is multiplicity-free.  We write $GL_{m|n}$, $OSP_{m|2n},P_{n|n}$, and $Q_{n|n}$ for the standard modules for each corresponding Lie superalgebra (assume $m,n>0$ in all cases).  For a representation $V$, write $\Pi V$ for the corresponding parity shift representation.

\renewcommand{\arraystretch}{1.5}

\[	\begin{tabular}{ |c|c|c| } 
\hline 
$V$ & $\text{dim}^sV$ & $S^\bullet V^*$ Completely Reducible? \\
\hline
$GL_{m|n}$ & $(m|n)$ & Always  \\
\hline 
$S^2GL_{m|n}$ & $(\frac{n(n-1)}{2}+\frac{m(m+1)}{2}|mn)$& Always\\ 
\hline
$\Pi S^2GL_{n|n}$ &  $(n^2|n^2)$& Always\\
\hline
$\Pi S^2GL_{n|n+1}$& $(n(n+1)|n(n+1))$& Always \\
\hline
$OSP_{m|2n}$, $m\geq 2$ & $(m|2n)$ &\makecell{Iff $m$ is odd \\ or $m>2n$} \\
\hline
$\Pi OSP_{m|2n}$ & $(2n|m)$ & Always \\
\hline
$Q_{n|n}$ & $(n|n)$ & Always\\
\hline 
$\Pi P_{n|n}$ & $(n|n)$ & Never  \\
\hline
\end{tabular}
\]
\renewcommand{\arraystretch}{1}

In addition to the above infinite families, we have some small cases of spherical indecomposable representations.  They are as follows.

\begin{itemize}
	\item For the Lie superalgebra $\g\l(1|2)$, write $\epsilon_1,\delta_1,\delta_2$ for a basis of its weight space.  Then the Kac modules $K_{1|2}(t\epsilon_1)$ are spherical for all $t$, and their parity shifts $\Pi K_{1|2}(t\epsilon_1)$ are spherical for $t\neq 0$.
	\item For $\o\s\p(2|4)$, write $\epsilon_1,\delta_1,\delta_2$ for a basis of its weight space, and consider the Borel $\b^{\epsilon\delta\delta}$ with simple roots $\epsilon_1-\delta_1,\delta_1-\delta_2,2\delta_2$.  Then $L_{\b^{\epsilon\delta\delta}}(\delta_1+\delta_2-\epsilon_1)$, the irreducible module of highest weight $\delta_1+\delta_2-\epsilon_1$ with respect to this Borel, is spherical.
	\item For $\p(3)$, write $\epsilon_1,\epsilon_2,\epsilon_3$ for the usual basis of its weight space, and let $\omega=\epsilon_1+\epsilon_2+\epsilon_3$.  Let $\nabla(\omega)$ denote the thin Kac module of highest weight $\omega$.  Then this module, along with its radical and quotient by its socle are all spherical.
	\item Finally, for $\q(2)$,  we may take $Q_{2|2}=L(\epsilon_1)$ with respect to the standard Borel, and twist the highest weight by multiples of $\epsilon_1+\epsilon_2$, and we will still have a dominant weight.  So we consider $(Q_{2|2})_t=L(\epsilon_1+t(\epsilon_1+\epsilon_2))$ for $t\in\C$.  If $t\neq -1/2$, then $(Q_{2|2})_t$ and $\Pi (Q_{2|2})_t$ are both spherical.  
	
	When $t=-1/2$, $(Q_{2|2})_{-1/2}$ is the representation obtained via the isomorphism $\q(2)/\C I_{2|2}\to[\p(2),\p(2)]$  and is the restriction of the standard representation of $\p(2)$.  We will therefore write this representation as $\text{Res}_{[\p(2),\p(2)]}P_{2|2}$.
\end{itemize}

	\

The following table of exceptional spherical indecomposable representations completes our list:

\renewcommand{\arraystretch}{1.5}
\[\begin{tabular}{ |c|c|c| } 
\hline 
$V$ & $\text{dim}^sV$ & $S^\bullet V^*$ Completely Reducible?\\
\hline
$U^{1|1}=\operatorname{Res}_{[\p(1),\p(1)]}P_{1|1}$ & $(1|1)$ & No\\
\hline
$K_{1|2}(t\epsilon_1)$ & $(2|2)$ & Iff $t\notin\Q\cap[0,1]$ \\
\hline
$\Pi K_{1|2}(t\epsilon_1)$, $t\neq 0$ & $(2|2)$ & Iff $t\neq 1$\\
\hline 
$L_{\b^{\epsilon\delta\delta}}(\delta_1+\delta_2-\epsilon_1)$ & $(6|4)$ & Yes \\
\hline
$\nabla(\omega)$ & $(4|4)$ & No\\
\hline
$\operatorname{rad}\nabla(\omega)$ & $(3|4)$ & No\\
\hline
$\nabla(\omega)/\operatorname{soc}\nabla(\omega)$ & $(4|3)$ & No\\
\hline
$(Q_{2|2})_{t},t\neq -1/2$ & $(2|2)$  & Always\\
\hline	
$\Pi(Q_{2|2})_t,t\neq -1/2$ & $(2|2)$ & Always\\
\hline
$\text{Res}_{[\p(2),\p(2)]}P_{2|2}$ & $(2|2)$ & No\\
\hline
\end{tabular}
\]
\renewcommand{\arraystretch}{1}

Note that there is some redundancy in the above list, in that some of the Kac modules for $\g\l(1|2)$ are equivalent and also some of the modules for $\q(2)$ showing up are equivalent.

\begin{remark}
Some observations about the above classification:
\begin{itemize}
	\item All cases with quasireductive stabilizer appear in \cite{sahi2018capelli}.
	\item The number of $\g_{\ol{0}}$-components of any spherical irreducible representation is always less than or equal to 3; in an indecomposable spherical representation, there may be up to 4 $\g_{\ol{0}}$-components.
	\item No exceptional basic lie superalgebras admit any non-trivial spherical representations.
\end{itemize} 

\subsection{Acknowledgments}  The author would like to thank his advisor, Vera Serganova, for the idea for the paper and useful discussions along the way.  The author also thanks the reviewer for helpful suggestions.  This research was partially supported by NSF grant DMS-1701532.

\end{remark}

\section{Notation, Definitions, and Background}\label{sec_defns}

\subsection{Notation List}
\begin{itemize}
	\item $V=V_{\ol{0}}\oplus V_{\ol{1}}$ is parity decomposition;
	\item $S^\bullet V$ symmetric algebra of $V$;
	\item $\g$ a quasireductive superalgebra;
	\item $\b$ a Borel subalgebra;
	\item $\n\sub\b$ the nilpotent subalgebra;
	\item $\b^{op}$ the opposite Borel of $\b$;
	\item $\h_{\ol{0}}$ a Cartan subalgebra of $\g_{\ol{0}}$;
	\item $\alpha$ a root of $\g$;
	\item $\Delta$ the set of all roots of $\g$;
	\item $\Sigma\sub\Delta$ a simple root system;
	\item $\h$ the centralizer of $\h_{\ol{0}}$ in $\g$.
	\item $\lambda$ a weight of $\h_{\ol{0}}$;
	\item $L(\lambda)=L_{\b}(\lambda)$ the irreducible representation of $\g$ of (even) dominant highest weight $\lambda$ with respect to $\b$;
	\item $L_0(\lambda)$ the irreducible representation of $\g_{\ol{0}}$ with respect to $\b_{\ol{0}}$ of (even) dominant highest weight $\lambda$ with respect to $\b_{\ol{0}}$;
	\item $r_\alpha$ an odd reflection;
	\item $r_{\alpha}(\b)$ the Borel gotten from $\b$ via $r_{\alpha}$;
	\item $\b^{\sigma}$ the Borel subalgebra defined by the $\epsilon\delta$-sequence $\sigma$ (see appendix \hyperref[App_borels]{A}).
	\item $K_{m|n}(\lambda)$ the Kac module of $\g\l(m|n)$ of dominant highest weight $\lambda$;
	\item $\g\ltimes V$ see second paragraph of \ref{basic_notation};
	\item $\g_{(\chi)}=\{X\in\g:\chi(X)=0\}$;
	\item $X$ a complex supervariety (see \ref{spher_var} for all other notation about complex supervarieties);

\end{itemize}

\subsection{Basic Notation:}\label{basic_notation}  We work in the category of super vector spaces over $\C$.  For a super vector space $V$ we write $V=V_{\ol{0}}\oplus V_{\ol{1}}$ for its parity decomposition.  For a homogeneous vector $v\in V$, we write $\ol{v}\in\Z_{2}=\{\ol{0},\ol{1}\}$ for its parity.  We also define $\dim^s V$ to be the ordered pair $(\dim V_{\ol{0}}|\dim V_{\ol{1}})$.  Write $\Pi V:=\C^{1|0}\otimes V$ for the parity shift of $V$.

Throughout, $\g=\g_{\ol{0}}\oplus\g_{\ol{1}}$ will denote a finite-dimensional Lie superalgebra.  If $\g$ is a Lie superalgebra and $V$ is a $\g$-module, we write $\g\ltimes V$ for the Lie superalgebra with underlying super vector space $\g\oplus V$ such that $V$ is an abelian ideal, $\g$ a subalgebra, and $[(X,0),(0,v)]=(0,Xv)$ for $X\in\g$, $v\in V$.  

For a Lie superalgebra $\g$ and a character $\chi$ of $\g$, we write $\g_{(\chi)}=\{X\in\g:\chi(X)=0\}$.

\subsection{Quasireductive Superalgebras; Borel Subalgebras} The study of quasireductive Lie superalgebras was developed in \cite{ivanova2008parabolic} and \cite{serganova2011quasireductive}.  We present the basic tools needed for the paper here, for clarity.
\begin{definition}
	A Lie superalgebra $\g$ is quasireductive if $\g_{\ol{0}}$ is a reductive Lie algebra, and the adjoint action of $\g_{\ol{0}}$ on $\g$ integrates to an action of a reductive algebraic group $G_{\ol{0}}\sub GL(\g)$.
\end{definition}

Let $\g$ be quasireductive. Choose a Cartan subalgebra $\h_{\ol{0}}$ of $\g_{\ol{0}}$ and consider the non-zero weights $\Delta\sub\h_{\ol{0}}^*$ of the action of $\h_{\ol{0}}$ on $\g$.  We refer to the pair $(\h_{\ol{0}},\Delta)$ as the root system of $\g$ (with respect to $\h_{\ol{0}}$), and refer to $\Delta$ as the roots.  Write $\Lambda\sub\h_{\ol{0}}^*$ for the weight lattice of $\g_{\ol{0}}$.  Then we have $\Delta\subseteq\Lambda$. Define
\[
\Lambda^\vee_{\R}=\{h\in\h_{\ol{0}}:\alpha(h)\in\R\text{ for all }\alpha\in\Delta\}
\]
Note that by our assumptions a real subspace of $\Lambda^\vee_{\R}$ will pair perfectly with $\operatorname{span}_{\R}\Delta$.

\begin{definition}
For a generic coweight $h\in\Lambda^\vee_{\R}$ (i.e. one such that $\alpha(h)\neq 0$ for all $\alpha\in\Delta$), define $\Delta^{\pm}=\{\alpha\in\Delta:\pm\alpha(h)>0\}$.    Define nilpotent subalgebras
\[
\n^{\pm}=\n^{\pm}(\h_{\ol{0}},h):=\bigoplus\limits_{\alpha\in\Delta^{\pm}}\g_{\alpha},
\]
and define $\h$ to be the centralizer of $\h_{\ol{0}}$ in $\g$, i.e. the zero weight space of the action of $\h_{\ol{0}}$.  In this case we say that a root $\alpha$ is positive (resp. negative) if $\alpha(h)>0$ (resp. $\alpha(h)<0$), and call $\Delta^{+}$ (resp. $\Delta^-$) the set of positive (resp. negative) roots.  

A Borel subalgebra $\b$ of $\g$ is defined to be a subalgebra of the form
\[
\b=\b(\h_{\ol{0}},h):=\h\oplus\n^+
\]
for some choice of $\h_{\ol{0}}$ and $h$.  In this case, we define 
\[
\b^{op}=\b(\h_{\ol{0}},h)^{op}:=\b(\h_{\ol{0}},-h)=\h\oplus\n^-
\]
and say $\b^{op}$ is the opposite Borel to, or of, $\b$.
\end{definition}

Let $\b$ be a Borel subalgebra of $\g$.  Then the following facts are well known.
\begin{itemize}
	\item $\b_{\ol{0}}$ is a Borel subalgebra of $\g_{\ol{0}}$.
	\item If $\tau$ is an automorphism of $\g$, then $\tau(\b)$ is again a Borel subalgebra.
\end{itemize}
\subsection{Representation Theory of $\g$}  Let $\g$ be a quasireductive Lie superalgebra.  We consider the category of finite-dimensional representations of $\g$ which are semi-simple over $\g_{\ol{0}}$.  The classification of irreducible representations in this category was explained in \cite{serganova2011quasireductive}.  We recall it here.  

Fix a Cartan subalgebra $\h_{\ol{0}}\sub\g$ and Borel $\b=\h\oplus\n^+$ containing $\h_{\ol{0}}$.  For a representation $V$ of $\g$, write $P(V)$ for the weights of $\h_{\ol{0}}$ with non-trivial weight spaces in $V$.  For two weights $\lambda,\mu\in\h_{\ol{0}}^*$, write $\mu\leq \lambda$ if we have $\lambda-\mu$ is a sum of positive roots of $\g$ with respect to $\b$. We refer to this as the Bruhat order on $\h_{\ol{0}}^*$ (with respect to $\b$).

The following facts are proven in \cite{serganova2011quasireductive}.

\begin{thm}\label{hw_thm}
	Let $L$ be a simple $\g$-module.  Then there exists a unique weight $\lambda\in P(L)$, called the highest weight of $L$ (with respect to $\b$), such that $\mu\leq\lambda$ for all $\mu\in P(L)$. Further, $L_{\lambda}$ is a simple $\h$-module.  If two simple $\g$-modules have the same highest weight, then they are isomorphic up to parity shift.
\end{thm}

\begin{cor}\label{cirrep_maxtorus}
	If $\h=\h_{\ol{0}}$ and $L$ is a simple $\g$-module of highest weight $\lambda$, there exists a unique (up to non-zero scalar) vector $v_{\lambda}\in L_{\lambda}$, which we call a highest weight vector of $L$.  
\end{cor}

\begin{definition}\label{irrep_def_max_torus}
	We say a weight $\lambda\in\h_{\ol{0}}^*$ is dominant with respect to $\b$ if there is a finite-dimensional irreducible representation $L$ of $\g$ of highest weight $\lambda$ with respect to $\b$. 
\end{definition}
%\begin{remark}
%	Note the above definitions of $L_{\b}(\lambda)$ are not standard.  In particular it does not, in general, agree with the decomposition of representations into blocks up to parity shift for basic simple algebras.  However for our purposes it will simplify notation.	
%\end{remark}

\textbf{Notation:} \begin{itemize}
	\item If $\h=\h_{\ol{0}}$ and $\lambda$ is dominant we define $L_{\b}(\lambda)$ (or simply $L(\lambda)$ when no confusion arises) to be a simple module of highest weight $\lambda$ with respect to $\b$ such that the highest weight vector is even.
	\item Because of subtleties surrounding parity, if $\h\neq \h_{\ol{0}}$ we write $L_{\b}(\lambda)$ (or $L(\lambda)$) for a fixed choice of irreducible representation of highest weight $\lambda$.  This case will only arise for us when $\g=\q(n)$, and this subtlety will be largely unimportant.
	\item In any case, we write $L_0(\lambda)$ for the even irreducible representation of $\g_0$ of highest weight $\lambda$.
\end{itemize}

\subsection{Odd Reflections in Basic Lie Superalgebras}\label{basic_lie}  Let $\g$ be one of the Lie superalgebras $\g\l(m|n)$, $\o\s\p(m|2n)$, $G(1,2)$, $F(1,3)$, or $D(1,2;\alpha)$.  We refer to these algebras as basic.  What distinguishes the basic algebras in our list is that they admit an even invariant non-degenerate supersymmetric form.  Basic superalgebras have many similarities to classical reductive Lie algebras.  We state what we will need for this paper.

Given a basic Lie superalgebra $\g$ and a Cartan subalgebra $\h_{\ol{0}}$ of $\g_{\ol{0}}$, we get a generalized root system $(\h_{\ol{0}}^*,\Delta)$, see \cite{serganova1996generalizations}.  The Weyl group $W$ of $\h_{\ol{0}}$ will act by symmetries on this root system.  Further, the Borels of $\g$ which contain $\h_{\ol{0}}$ will be in bijection with choices of simple roots in $\Delta$, as in the classical case.  

If $\Sigma\sub\Delta$ is a set of simple roots, and $\alpha\in\Sigma$ is isotropic, then we denote by $r_{\alpha}$ the odd reflection with respect to $\alpha$, which takes $\Sigma$ to a new simple root system $r_{\alpha}(\Sigma)$.  See section 1.4 of \cite{cheng2012dualities} or the axiomatic approach in section 1 of \cite{serganova1996generalizations} for more on how odd reflections change the simple root system.

The main result we will be using regarding simple reflections is the following: let $\b$ be the Borel corresponding to $\Sigma$, and $r_{\alpha}(\b)$ the Borel corresponding to $r_{\alpha}(\Sigma)$.  Further let $\lambda\in\h^*$ be dominant with respect to $\b$. Then we have:
\[
L_{\b}(\lambda)\cong L_{r_{\alpha}(\b)}(\lambda) \ \ \text{ if }(\lambda,\alpha)=0\ \text{ and } \ L_{\b}(\lambda)\cong \Pi L_{r_{\alpha}(\b)}(\lambda-\alpha)\ \ \text{ if } (\lambda,\alpha)\neq0
\]
(see \cite{cheng2012dualities}, lemma 1.40 for a proof and sections 1.4 and 1.5 for more on odd reflections in highest weight theory). 
\subsection{Spherical Supervarieties}\label{spher_var}

See chapter 10 of \cite{carmeli2011mathematical} for generalities on superschemes.

\begin{definition}
	Define a (complex) affine supervariety to be an affine superscheme $X=\Spec A$ for a finitely-generated commutative superalgebra $A=A_{\ol{0}}\oplus A_{\ol{1}}$ over $\C$ such that if $f\in A\setminus(A_{\ol{1}})$, then $f$ is not a zero divisor.  In particular, $\Spec A/(A_{\ol{1}})$ is integral.
	
Here, $(A_{\ol{1}})$ is the ideal of $A$ generated by all odd elements $A$.  We will also use the notation $A=\C[X]$ for the algebra of functions on $X$.

If $X$ is an affine supervariety, we will write $X(\C)$ for its set of closed points.
\end{definition}

\begin{example}
 If $V$ is a super vector space then we may consider it as an affine supervariety with coordinate ring  $S^\bullet V^*$.
\end{example}

\begin{definition}
	Let $X$ be an affine supervariety, $x\in X(\C)$.  Define $T_xX$ to be the super vector space of point-derivations at $x$, i.e. all (not-necessarily even) linear maps $\delta:\C[X]\to\C$ such that $\delta(fg)=\delta(f)g(x)+(-1)^{\ol{\delta}\cdot \ol{f}}f(x)\delta(g)$.
\end{definition}

\begin{definition}Let $\g$ be a Lie superalgebra, $X$ an affine supervariety.  An action of $\g$ on $X$ is an algebra homomorphism $\g\to\text{Vec}(X)$, where $\text{Vec}(X)$ is the Lie superalgebra of vector fields on $X$.  In this case, we say that $\g$ acts on $X$.  
\end{definition}

\begin{remark}
	In many ways it is more natural to consider the action of a Lie supergroup on a variety, which in particular gives rise to a Lie superalgebra action as described above.  However for the purposes of this paper it is not necessary.  In future work we will consider actions of supergroups to obtain stronger results.
\end{remark}
\begin{definition}
For an open subvariety $U$ of $X$, we say that $U$ is a $\g$-orbit if for every closed point $x\in X(\C)$, the natural restriction map
\[
\g\to T_xX
\]
is surjective if and only if $x\in U(\C)$.  We say that $\g$ has an open orbit on $X$ if there exists an open subvariety $U$ of $X$ which is an open orbit for $\g$.  Note that by lower semicontinuity of the rank of the map $\g\to T_xX$, $\g$ has an open orbit on $X$ if and only if $\g\to T_xX$ is surjective for some closed point $x$.

If $x\in X(\C)$, we define the stabilizer of $x$ in $\g$ to be $\ker(\g\to T_xX)$.  This is a subalgebra of $\g$.
\end{definition}

\begin{definition} Let $\g$ be a quasireductive Lie superalgebra acting on $X$, $\b$ a Borel subalgebra of $\g$.  Then we say that $X$ is \emph{spherical} for the action of $\g$ with respect to $\b$, if $\b$ has an open orbit on $X$.  We will sometimes simply say $X$ is spherical for $\g$, or even $X$ is spherical, if $\g$ and $\b$ are clear from context.
\end{definition}

We would like a representation-theoretic characterization of when a given (affine) variety is spherical.  Classically, for an affine variety, sphericity is equivalent to the algebra of functions being multiplicity-free.  The following example shows why this is not sufficient in the super world.

\begin{example}
	Consider the space $GL_{0|n}$ with the natural $\g\l(n)$-action.  Then the algebra of functions is, as a $\g\l(n)$-module, $\bigoplus\limits_{k}\Pi^k\Lambda^kGL_{n}^*$.   This is a multiplicity-free $\g\l(n)$-module, however the space is not spherical.  The issue, as the following proposition indicates, is the existence of nilpotent highest weight functions.
\end{example}

The following result is proven in \cite{sherman2019sphericalsupervar}.

\begin{prop}\label{pspher_char} Let $\g$ be a quasireductive Lie superalgebra acting on an affine supervariety $X$.  Suppose that $X_{\ol{0}}$ is a spherical $\g_{\ol{0}}$-variety.  Let $\b$ be a Borel subalgebra of $\g$.  Then $X$ is spherical for $\g$ with respect to $\b$ only if for each $\b$-highest weight submodule $V$ of $\C[X]$, there is a non-nilpotent function of highest weight in $V$.    

In particular, if $\g$ has a maximal even torus then $X$ is spherical for $\g$ with respect to $\b$ only if all $\b$-highest weight vectors are non-nilpotent functions.
\end{prop}

\begin{cor}\label{cmult_free} In the context of \cref{pspher_char}, if $\h=\h_{\ol{0}}$ then the socle of $\C[X]$ is a multiplicity-free $\g$-module.\end{cor}

\begin{proof}
Suppose that $L$ is an irreducible $\g$-module of highest weight $\lambda$ with multiplicity greater than one in the socle of $\C[X]$.  Using \cref{pspher_char}, we find there exists two linearly independent highest weight functions $f,g$ on $\C[X]$ of weight $\lambda$ which are non-nilpotent.  Therefore, they each restrict to $X_{\ol{0}}$ as non-zero functions which are $\b_{\ol{0}}$-eigenfunctions of weight $\lambda$.  Since $X_{\ol{0}}$ is spherical, there is only one such non-zero function on $X_{\ol{0}}$ up to scalar, and therefore $f-cg$ must be a nilpotent highest weight function on $X$ for some $c\in\C$.  By \cref{pspher_char}, this implies $f=cg$, a contradiction.
\end{proof}

%\begin{definition}\label{monoid_def}
%	If $X$ is an affine spherical variety for $\g$ with respect to $\b$, we write $\Lambda^+_\b(X)$ (or simply $\Lambda^+(X)$) for the collection of $\b$-highest weights of $\C[X]$.  By \cref{pspher_char}, it is in fact a monoid.
%\end{definition}
%
%\begin{remark}In the case when $\h$ is not an even maximal torus, e.g. $\h=\q(1)\times\q(1)\times\q(1)$, it may be that the socle of $\C[X]$ is not multiplicity-free.   When such a circumstance occurs, there will be two linearly independent non-nilpotent functions of highest weight $\lambda$.  Therefore, the monoid as we have defined it will not contain all the information about highest weight functions which appear.
%
%Nevertheless, such a circumstance will not arise for any of the spaces we study.
%\end{remark}

\section{Spherical Representations and their Properties}\label{sec_tools}
	
\subsection{Spherical Representations}  	
\begin{definition}\label{spher_reps}
	Let $V$ be a super vector space, $\g$ a quasireductive Lie superalgebra, $\b\sub\g$ a Borel subalgebra of $\g$, and $\rho:\g\to\g\l(V)$ a representation of $\g$.  Then we say the triple $(V,\g,\rho)$ is \emph{spherical} with respect to $\b$ if $V$ is a non-zero vector space, and it is a spherical variety for the Lie superalgebra $\rho(\g)+\C\id_{V}$ with respect to $\rho(\b)+\C\id_{V}$.  
	
	Equivalently, there exists a vector $v\in V_{\ol{0}}$ such that $(\rho(\b)+\C\id_{V})\cdot v=V$.  A vector $v$ satisfying this condition will be called a \emph{spherical vector}.
	
	In general we say a $\g$-module $V$ is spherical if there exists a Borel $\b\sub\g$ such that $(V,\g,\rho)$ is spherical with respect to $\b$, where $\rho:\g\to\g\l(V)$ defines the $\g$-action.  
\end{definition} 

\begin{remark}
	By abuse of language, we will often refer to a super vector space $V$ as being spherical when the algebra which acts on it is clear from context. We may also omit the representation $\rho$ and just say that $(V,\g)$ is spherical, where $V$ is a $\g$-module and the action is clear from context.
	
	If $V$ is spherical of even (resp. odd) highest weight $\lambda$ with respect to $\b$, then we will say that $\lambda$ is an even (resp. an odd) spherical weight for $\g$ with respect to $\b$, or simply that $\lambda$ is spherical when the choice of Borel and parity of the highest weight is clear from context.  
\end{remark}

\begin{remark}
	In contrast to the classical world, note that a spherical $\g$-module is, in general, not spherical for all Borels (this almost never happens		).  However, if $V$ is spherical for $\g$ with respect to $\b$, then it is also spherical with respect to any conjugate of $\b$.
\end{remark}

\begin{lemma}\label{lgeoequiv} Let $(V,\g,\rho)$ be a spherical representation with respect to $\b$.   
	\begin{enumerate}
		\item If $V'$ is a super vector space and $\psi:V\to V'$ is an isomorphism, let $\Psi:\g\l(V)\to\g\l(V')$ denote the induced isomorphism of algebras.  Then $(V',\g,\Psi\circ\rho)$ is spherical with respect to $\b$.
		
		\item Let $\tau$ be an automorphism of $\g$. Then $(V,\g,\rho\circ\tau)$ is spherical with respect to $\tau^{-1}(\b)$.
	\end{enumerate}
	
\end{lemma}
\begin{proof} 
The proof is straightforward.
\end{proof}
By \cref{lgeoequiv}, sphericity is determined by the image of a Lie superalgebra under the representation.  Therefore we make the following definition.
\begin{definition}\label{equiv_def}
We say that two spherical representations $(V,\g,\rho)$ and $(V',\g',\rho')$ are \emph{equivalent} if there exists an isomorphism of super vector spaces $\psi:V\to V'$ such that if $\Psi:\g\l(V)\to\g\l(V')$ is the induced map, then $\rho'(\g')+\C\id_{V'}=\Psi(\rho(\g))+\C\id_{V'}$.
\end{definition}

We may now state the main problem this paper tries to solve:

\textbf{Problem}: Find all equivalence classes of indecomposable spherical representations $V$ which are spherical with respect to one of the Lie superalgebras $\g\l(m|n)$, $\o\s\p(m|2n)$, simple basic exceptional, $\p(n)$, or $\q(n)$.  We only look for cases when $V_{\ol{1}}\neq 0$.

\subsection{Properties of Spherical Representations}  Here we collect some properties of spherical modules.  First a definition.

\begin{definition}
	We say $(V,\g,\rho)$ is \emph{numerically spherical} if $V_{\ol{0}}$ is a spherical $\g_{\ol{0}}$-module and $\dim V_{\ol{1}}\leq \max_{\b}\dim\b_{\ol{1}}$.  
\end{definition}

\begin{lemma}\label{ldim_res} Spherical modules are numerically spherical.  Any composition factor of a spherical module is numerically spherical.\end{lemma}
\begin{proof} For the first statement, if $(V,\g,\rho)$ is spherical with respect to $\b$, then if $v\in V_{\ol{0}}$ is a spherical vector we have 
	\[
	(\b+\C\id_{V})\cdot v=(\b_{\ol{0}}+\C\id_V+\b_{\ol{1}})\cdot v=(\b_{\ol{0}}+\C\id_{V})\cdot v\oplus \b_{\ol{1}}\cdot v=V_{\ol{0}}\oplus V_{\ol{1}}.
	\]
	The second statement is straightforward.	
\end{proof}

\begin{remark}
	Note that if $\g$ is basic, then all Borels of $\g$ have the same odd dimension, while in general this is no longer true.  In particular, for $\p(n)$ Borel subalgebras can have odd dimension between $\frac{n(n-1)}{2}$ and $\frac{n(n+1)}{2}$.
\end{remark}

\begin{lemma}\label{quot} The quotient of a spherical representation remains spherical.\end{lemma}
\begin{proof}The image of a spherical vector under the quotient map provides a spherical vector in the quotient.\end{proof}

\begin{lemma}\label{cdual}
	If $\g$ is basic, then an irreducible representation $V$ is spherical if and only if $V^*$ is.  If $(V,\g)$ is spherical, then $(V^*,\g)$ is equivalent to it.
\end{lemma}
\begin{proof}
	In this case, there exists an automorphism $\tau$ of $\g$ which acts by multiplication by $(-1)$ on a Cartan subalgebra.  Since highest weights spaces of irreducible representations are one-dimensional for a basic algebra, if $V$ is irreducible we have, $V^\tau\cong V^*$.  The result now follows from \cref{lgeoequiv}.
\end{proof}
\begin{remark}
	If we drop the condition of irreducibility in \cref{cdual}, the argument breaks down, since we no longer have $V^\tau=V^*$. 
	
	For example, let $\g=\g\l(1|1)$ and consider the representation of $\g$ on $\C^{1|1}=\C v\oplus \C w$, where $v$ is even and $w$ odd, as follows.  Let $I=\begin{bmatrix}1 & 0\\0 & 1\end{bmatrix}$ act by $0$, $h=\begin{bmatrix} 1 & 0\\0 & -1\end{bmatrix}$ act by $1$ on $w$ and $-1$ on $v$, and $x=\begin{bmatrix}0 & 1\\0 & 0\end{bmatrix}$ send $v$ to $w$ and $y=\begin{bmatrix}0 & 0 \\1 & 0\end{bmatrix}$ act by $0$.
	
	Then this is representation is spherical (with respect to the Borel $\b=\C\langle h,I,x\rangle$), but the dual is not spherical with respect to any Borel subalgebra.
\end{remark}

\begin{remark}
	By \cref{cdual}, when $\g$ is basic it suffices to consider irreducible representations up to their dual.  Note that \cref{cdual} does not apply to the non-basic algebras $\p(n)$ or $\q(n)$. Although $\q(n)$ does admit an automorphism $\tau$ which acts by $(-1)$ on $\h_{0}$, because highest weight spaces of irreducible representations need not be one-dimensional,  there are situations when $V$ is irreducible but $V^\tau\ncong V^*$.
\end{remark}

\begin{lemma}\label{lspher_vec} Let $\r$ be a linear Lie superalgebra, i.e. $\r\sub\g\l(V)$ for a super vector space $V$.  Suppose we have $\r_{\ol{0}}v=V_{\ol{0}}$.  If there exists $w\in V_{\ol{0}}$ such that $\r\cdot w=V$, then $\r\cdot v=V$.  
\end{lemma}
\begin{proof}
	Exponentiate $\r_{\ol{0}}$ to $R_{\ol{0}}=\exp(\r_{\ol{0}})\sub GL(V)_{\ol{0}}$.  Then by assumption, the orbits of both $v$ and $w$ under $R_{\ol{0}}$ are open in $V_{\ol{0}}$ and since we work in the Zariski topology they must lie in the same orbit. Therefore there exists $x\in R_{\ol{0}}$ such that $v=x\cdot w$.  Hence 
	\begin{eqnarray*}
	\r\cdot v& = &\{Xv:X\in\r\}\\
	         & = &\{X(xw):X\in\r\}\\
	         & = &\{x\Ad(x^{-1})(X)w):X\in\r\}\\
	         & = &\{x(Xw):X\in\r\}\\
	         & = &x(\r\cdot w)=V
	\end{eqnarray*}
\end{proof}

\begin{cor}\label{cspher_vec}
	Let $V$ be a $\g$-module, $\b$ a Borel of $\g$ such that $V_{\ol{0}}$ is spherical for $\g_{\ol{0}}$.  Then if $v\in V_{\ol{0}}$ is any spherical vector for $\b_{\ol{0}}$ and $\b\cdot v\neq V$, then $V$ is not spherical with respect to $\b$.
\end{cor}
\begin{proof}
	This follows from \cref{lspher_vec} by letting $\r=(\rho(\b)+\C\id_{V})$, where $\rho:\g\to\g\l(V)$ is the representation giving the $\g$-action.
\end{proof}

\begin{lemma}\label{spher_rep_res_1}
	Let $V$ be an irreducible spherical $\g$-module of $\b$-highest weight $\lambda$, such that $\dim V_{\lambda}=(0|1)$.  We do not assume $V$ is spherical with respect to $\b$.  Then there exists an odd negative root $\alpha$ such that $\lambda+\alpha$ is a $\b_{\ol{0}}$-highest weight of $V_{\ol{0}}$.  In particular, $\lambda+\alpha$ is $\g_{\ol{0}}$-spherical.
\end{lemma}
\begin{proof}
	Let $v\in V_{\lambda}$ be a non-zero highest weight vector.  Consider the set 
	\[
	S=\{\alpha\in\Delta^{-}:x_{\alpha}\cdot v\neq 0\text{ for some }x_{\alpha}\in\g_{\alpha}\}
	\]
	Then we claim $S\neq\emptyset$.  If not, $v$ will be annihilated by $\n^+_{\ol{1}}$ and $\n^-_{\ol{1}}$.  The condition that $\dim V=(0|1)$ implies further that $v$ is annihilated by $\h_{\ol{1}}$, so we find then that $v$ is annihilated by $\g_{\ol{1}}$.  Since $V=\UU\g\cdot v$, this in turn implies that $V_{\ol{0}}=0$, a contradiction.
	
	Since $S$ is not empty, we may choose $\alpha\in S$ which is maximal with respect to the Bruhat order.  Then for $z\in\n^+_{\ol{0}}$ we have
	\[
	zx_{\alpha}v=[z,x_{\alpha}]v+x_{\alpha}zv=[z,x_{\alpha}]v
	\]
	However, $[z,x_{\alpha}]$ is of weight strictly larger than $\alpha$ in the Bruhat order.  If $[z,x_{\alpha}]\in\n^{-}$, then by maximality of $\alpha$ we have $[z,x_{\alpha}]v=0$.  If $[z,x_{\alpha}]\in\h_{\ol{1}}\oplus\n_{\ol{1}}^+$, then we also have $[z,x_{\alpha}]v=0$.  So $x_{\alpha}v$ is an even $\b_{\ol{0}}$-highest weight vector.
\end{proof}

\begin{lemma}\label{spher_rep_restrictions}
  Suppose $\g$ is quasireductive with even Cartan subalgebra $\h=\h_{\ol{0}}$. Let $\lambda$ be a dominant highest weight.  If $L(\lambda)$ is numerically spherical, then the following hold:
  \begin{enumerate}
  	\item $L_{0}(\lambda)$ is a spherical $\g_{\ol{0}}$-module;
  	\item If $\g$ is basic, and $\alpha$ is a simple isotropic root such that $(\lambda,\alpha)\neq0$, then $\dim L_{0}(\lambda-\alpha)\leq \dim\b_{\ol{1}}$.
  \end{enumerate}
  \end{lemma}
  
  \begin{proof}
  $(1)$ follows from the observation that $L(\lambda)_{\ol{0}}$ must be a spherical $\g_{\ol{0}}$-module, and that $L_{0}(\lambda)$ is a $\g_{\ol{0}}$-submodule of $L(\lambda)_{\ol{0}}$.	 For $(2)$, by the theory of odd	 reflections we will have that $L_0(\lambda-\alpha)$ is a $\g_{\ol{0}}$-submodule of $L(\lambda)_{\ol{1}}$, and the statement follows.
    \end{proof}
  
  \begin{lemma}\label{spher_rep_restrictions_2}
 In the context of \cref{spher_rep_restrictions}, if $\Pi L(\lambda)$ is numerically spherical, then the following hold:
  \begin{enumerate}
  	\item $\dim L_{0}(\lambda)\leq\max\limits_{\b}\dim\b_{\ol{1}}$; 
  	\item If $\g$ is basic, and $\alpha$ is a positive simple isotropic root such that $(\lambda,\alpha)\neq0$, then $L_{0}(\lambda-\alpha)$ is a spherical $\g_{\ol{0}}$-module.
  	\item If $\g$ is basic and $\lambda$ is not a character, there exists a positive odd isotropic root $\alpha$ such that $(\lambda,\alpha)\neq0$ and $L_{0}(\lambda-\alpha)$ is a spherical $\g_{\ol{0}}$-module.
  \end{enumerate}
\end{lemma}

\begin{proof}

(1) follows from by definition of numerically spherical modules.

For $(2)$, by the theory of odd reflections $L_0(\lambda-\alpha)$ will be a $\g_{\ol{0}}$-submodule of $L(\lambda)_{\ol{0}}$, and the statement follows.

For (3), since $\lambda$ is not character, if we consider all possible sequences of odd reflections we can apply to our Borel, there will be some sequence  of odd reflections $r_{\alpha_s},\cdots,r_{\alpha_1}$ giving rise to a new Borel such that $(\lambda,\alpha_1)=\cdots=(\lambda,\alpha_{s-1})=0$ and $(\lambda,\alpha_s)\neq 0$.  

Let $\alpha=\alpha_s$.  By the theory of odd reflections, $\lambda-\alpha$ will be an even highest weight vector with respect this new Borel of $\g$.  Therefore $L_{0}(\lambda-\alpha)$ will be a spherical $\g_{\ol{0}}$-module.
\end{proof}

\begin{remark}\label{char_rem}
Characters of a Lie superalgebra, including the trivial one, are all spherical and equivalent to $(\C^{1|0},\mathbf{0})$ where $\mathbf{0}$ denotes the trivial Lie algebra.  Further, if $\chi$ is a character of $\g$ and $(V,\g,\rho)$ is a spherical representation, then $(V,\g,\rho\otimes \chi)$ is also spherical and is equivalent to $(V,\g,\rho)$.  It follows we may work with equivalence classes of spherical representations up to twists by characters.

We also observe that all one-dimensional odd modules are numerically spherical.  
\end{remark}

\section{Explanation of Procedure for Proof}\label{proof_procedure}

To classify spherical indecomposable representations, we will work case by case with various Lie superalgebras.

For a chosen algebra $\g$, we will first list any needed notation and setup. Then we will determine all (numerically) spherical irreducible representations.  If $\g$ has that $\h=\h_{\ol{0}}$ (i.e. for all cases but $\g=\q(n)$), we proceed according to the following steps:

\begin{enumerate}
	\item We choose a fixed, `standard' Borel for $\g$, which we write as $\b^{st}$. 
	\item We state all $\g_{\ol{0}}$-dominant weights which are spherical with respect to $\b^{st}_{\ol{0}}$. We will simply quote the results found by Kac in \cite{kac1980some}.
	\item From (1) and (2), we write a list of candidate $\b^{st}$-dominant weights $\lambda$ for which $L(\lambda)$ could be numerically spherical.  By \cref{char_rem}, we may find candidate weights up to twists by characters of $\g$. 
	
	Such $\lambda$ must have the properties that they are dominant with respect to $\b^{st}$ and spherical with respect to $\g_{\ol{0}}$.  Further, by \cref{spher_rep_restrictions}, if $\g$ is basic and $\alpha$ is a simple positive isotropic root such that $(\lambda,\alpha)\neq0$, then we must have $\dim L_0(\lambda-\alpha)\leq \dim\b_{\ol{1}}$.  This will be used heavily.  
	\item Determine whether $L(\lambda)=L_{\b^{st}}(\lambda)$ is (numerically) spherical for each $\lambda$ from (3).
	\item Create a candidate list of $\b^{st}$-dominant weights $\lambda$ for which $\Pi L(\lambda)$ could be numerically spherical.  Again, we can work up to twists by characters.  By \cref{spher_rep_restrictions_2}, such a $\lambda$ must have that 
	\[
	\dim L_{0}(\lambda)\leq\max\limits_{\b}\dim\b_{\ol{1}}.
	\]
	
	If $\g$ is basic, then \cref{spher_rep_restrictions_2} says that $\lambda-\alpha$ falls into the list in (2) for some positive (with respect to $\b^{st}$) isotropic root $\alpha$ with $(\lambda,\alpha)\neq0$.  Further, if $\alpha$ is a simple positive isotropic root such that $(\lambda,\alpha)\neq0$, then $\lambda-\alpha$ falls into the list in (2). This will be used heavily.
	
	If $\g=\p(n)$, then by \cref{spher_rep_res_1} there must exist an odd negative root $\alpha$ for which $\lambda+\alpha$ is spherical for $\g_{\ol{0}}$.
	\item Determine whether $\Pi L(\lambda)$ is (numerically) spherical for each $\lambda$ from (5).
\end{enumerate}

If $\g=\q(n)$, then we proceed as above, except we make a single list of candidate weights for which $L(\lambda)$ or $\Pi L(\lambda)$ could be (numerically) spherical, and then make a check.  This is because unless $\lambda=0$, we will have $L(\lambda)_{\lambda}=(k|k)$ where $k>0$, so some conditions in both (3) and (5) will apply to $\lambda$.

\

The above steps will give all (numerically) spherical irreducible representations.  By \cref{ldim_res} and \cref{quot}, an indecomposable spherical representation must have numerically spherical composition factors and a spherical head.  To determine all spherical indecomposables, we will take the modules from our above list and compute extensions between them, and check if any extensions are spherical.

\section{Spherical $\g\l(m|n)$ Modules}\label{gl_case}

Set $\g=\g\l(m|n)$.  Using $\g\l(m|n)\cong\g\l(n|m)$, in finding all spherical representations we may assume without loss of generality that $m\leq n$.  For more on this algebra, see \cite{musson2012lie} or \cite{cheng2012dualities}.  We refer the reader to appendix \hyperref[App_borels]{A} for an explanation of our notation for the root system and Borels of $\g$.  

\textbf{Notation:} Write
\[
{\det}_{\epsilon}:=\epsilon_1+\dots+\epsilon_m,\ \ \ {\det}_{\delta}:=\delta_1+\dots+\delta_{n}
\]
For a $\g_{\ol{0}}$-dominant weight $\lambda$, we will write $K_{m|n}(\lambda)$ for the Kac module
\[
K_{m|n}(\lambda)=\UU\g\otimes_{\UU(\g_{0}\oplus\g_{1})}L_{0}(\lambda).
\]
Here we use the usual $\Z$-grading on $\g$, $\g=\g_{-1}\oplus\g_0\oplus\g_{1}$.

We write $GL_{m|n}$ for the standard $\g\l(m|n)$-module structure on $\C^{m|n}$.  

\

\textbf{Choice of standard Borel:} $\b^{st}=\b^{\epsilon^m\delta^n}$, i.e. the Borel of upper triangular matrices of $\g\l(m|n)$.  Observe that for every Borel $\b$ of $\g$, we have $\dim\b_{\ol{1}}=mn$.  

%\
%
%\textbf{$\b^{st}$-dominant weights}: Recall that the dominant weights with respect to $\b^{st}$ are exactly those which are dominant with respect to $\b_{\ol{0}}$ in $\g_{\ol{0}}$.  Explicitly, dominant weights are exactly those of the form 
%	\[
%	s_1\epsilon_1+\dots+s_m\epsilon_m+t_1\delta_1+\dots+t_n\delta_n
%	\]where $s_i,t_j\in\C$ and $s_i-s_{i+1},t_j-t_{j+1}\in\Z_{\geq0}$ for all $i,j$.
%
\

\textbf{Characters of $\g$}: The characters of $\g$ are exactly the multiples of the Berezinian weight $\Ber$, defined by
\[
\Ber=\epsilon_1+\dots+\epsilon_m-\delta_1-\dots-\delta_n={\det}_{\epsilon}-{\det}_{\delta}
\] 

\

\textbf{Spherical weights for $\g_{\ol{0}}=\g\l(m)\times\g\l(n)$}:
\[
0, \ \ \epsilon_1,\ \ -\epsilon_m,\ \ \delta_1, \ \ -\delta_n
\]
\begin{align}\label{gl_even}
2\epsilon_1,\ \ -2\epsilon_m,\ \ 2\delta_1, \ \ -2\delta_n
\end{align}
\[
\epsilon_1+\epsilon_2,\ \ -\epsilon_{m-1}-\epsilon_m, \ \ \delta_1+\delta_2, \ \ -\delta_{n-1}-\delta_n
\]
\[
\epsilon_1+\delta_1,\ \ \epsilon_1-\delta_n,\ \ -\epsilon_m+\delta_1, \ \ -\epsilon_m-\delta_n 
\]
We also may add to any of the above weights an arbitrary linear combination of $\det_{\epsilon}$ and $\det_{\delta}$, and get another spherical weight.

\

\subsection{The Case of $\g\l(1|1)$}

We deal with $\g=\g\l(1|1)$ separately.  We have
\[
\g_{\ol{1}}=\C\langle u_{+},u_{-}\rangle, \text{ where }u_{\pm}\in\g_{\pm(\epsilon_1-\delta_1)}
\] 
There are exactly two Borel subalgebras, and they are non-conjugate:
\[
\b^{st}:=\b^{\epsilon\delta}=\g_{\ol{0}}\oplus\C\langle u_{+}\rangle,\ \ \b^{\delta\epsilon}=\g_{\ol{0}}\oplus\C\langle u_{-}\rangle
\]
Every weight $s\epsilon_1+t\delta_1$ is dominant.  These are all $\g_{\ol{0}}$-spherical, and all differ with $t\epsilon_1$ by some multiple of the Berezinian, so without loss of generality we can restrict our attention to weights of the form $t\epsilon_1$.
\begin{prop}
	For $\g\l(1|1)$, all non-trivial indecomposable spherical modules (up to equivalence) fall into the following list:
	\begin{enumerate}
		\item The standard module $GL_{1|1}$, which is spherical with respect to $\b^{\delta\epsilon}$ and has stabilizer $(\b^{\epsilon\delta})_{(\epsilon_1)}$ 
		\item The $(1|1)$-dimensional module $K_{1|1}(0)$, which is spherical with respect to $\b^{\delta\epsilon}$, and has stabilizer $\b^{\epsilon\delta}$.  This module is equivalent to the $\p(1)$-module $P_{1|1}$ (see \cref{p_case}).
	\end{enumerate}
\end{prop}

\begin{proof} The proof is straightforward and hence omitted.\end{proof}

\

\subsection{The Case of $\g\l(1|2)$}\

\

\textbf{Candidate even weights:}
\[
t\epsilon_1\ (t\neq 0),\ \ \ \ -\epsilon_1+\delta_1,\ \ \ \ -2\epsilon_1+2\delta_1
\]

\

\textbf{Check for (numerical) sphericity of $L(\lambda)$:}
\begin{itemize}
	\item  First suppose that $\lambda=t\epsilon_1$, with $t\neq 0$.  In this case, $L(\lambda)$ is a quotient of $K_{1|2}(\lambda)$, which is $(2|2)$-dimensional.  It's a straightforward check that $K_{1|2}(\lambda)$ is always spherical, and so $L(\lambda)$ is also.
	
	\item If $\lambda=-\epsilon_1+\delta_1$, then $\lambda$ differs with the Berezinian by $-\delta_2$, which gives $\Pi GL_{1|2}^*$ which is spherical (this will be shown later).
	
	\item  Finally suppose $\lambda=-2\epsilon_1+2\delta_1$. Then $\lambda$ differs by a multiple of the Berezinian with $-2\delta_2$, which is the highest weight of $\Lambda^2GL_{1|2}^*$.  We will see below that the second exterior power of the standard module for $\g\l$ is always spherical, so this is too.
\end{itemize}

\

\textbf{Candidate odd weights:}
\[
t\epsilon_1 \ (t\neq 0), \ \ \ \ \lambda=t\epsilon_1+\delta_1
\]

\
  
\textbf{Check for (numerical) sphericity of $\Pi L(\lambda)$:}
\begin{itemize}
	\item[$\bullet$]Suppose $\lambda=t\epsilon$ with $t\neq0$.  Then if $t=1$, we get the parity shift of the standard module, which is spherical (shown below).  If $t\neq 1$ then $\lambda$ is a typical weight, so that $\Pi L(\lambda)\cong \Pi K_{1|2}(\lambda)$.  One can check that this module is spherical exactly with respect to the Borel $\delta\epsilon\delta$. 
	
	\item[$\bullet$] If $\lambda=t\epsilon_1+\delta_1$, then if $t=-1$, $\Pi L(\lambda)$ is equivalent up to the Berezinian with $GL_{1|2}^*$, which is spherical.  If $t=1$, then $\lambda=\epsilon_1+\delta_1$ which gives $\Lambda^2GL_{1|2}$, and as already mentioned this is spherical.  If $t\neq\pm1$, then the sequence of odd reflections $r_{\epsilon_1-\delta_1}$ followed by $r_{\epsilon_1-\delta_2}$ both change the weight, which forces the odd dimension of the module to be larger than 2, so it cannot be numerically spherical.  
\end{itemize}

\

\textbf{(Numerically) Spherical irreducibles for $\g\l(1|2)$:} Along with $(\Pi)\C$ we have the following, up to equivalence:
\[
(\Pi)GL_{1|2}, \ \ \ (\Pi)K_{1|2}(t\epsilon_1)\ (t\neq 0,1)\ \ \ \Lambda^2GL_{1|2}
\]

\

\textbf{Indecomposables} For dimension reasons, the only possible extensions which are spherical are of an even one-dimensional module with a spherical irreducible or of $\Pi GL_{1|2}$ with a module equivalent to it or an odd one-dimensional module.  

The trivial module for $\g\l(1|2)$ admits non-trivial extensions only with $(GL_{1|2})_{\Ber}$ and $(GL_{1|2})^*_{-\Ber}$, and these extensions are exactly $K_{1|2}(0)$, $K_{1|2}(0)^*\cong (K_{1|2}(\epsilon_1))_{\Ber}$ along with two modules that are geometrically equivalent to these.  It has already been noted that $K_{1|2}(0)$ and $K_{1|2}(\epsilon_1)$ are spherical, so we get two new indecomposable spherical modules in this way.  The module $K_{1|2}(0)$ is spherical exactly with respect to $\delta\delta\epsilon$ and $\epsilon\delta\delta$, while $K_{1|2}(\epsilon_1)$ is spherical exactly with respect to only $\delta\delta\epsilon$.  No further extensions can be constructed which remain spherical.

The module $\Pi GL_{1|2}$ has extensions exactly with $\Pi \C_{-\Ber}$ and $\Lambda^2GL_{1|2}$.  The extensions by the latter is not numerically spherical, and hence not spherical.  But the extension with $\Pi \C_{-\Ber}$ with $\Pi GL_{1|2}$ as the quotient is exactly $\Pi K_{1|2}(\epsilon_1)$, which is spherical exactly with respect to the Borel $\delta\epsilon\delta$.  Note the opposite extension of the two has an odd one-dimensional quotient, so cannot be spherical.  No further extensions can be constructed which remain spherical.

Our work shows we get the following list of spherical indecomposables for $\g\l(1|2)$ which do not come from the standard module, along with stabilizers of spherical vectors and Borels for which sphericity is achieved.  All modules arising from the standard module will be dealt with in the next subsection.
\renewcommand{\arraystretch}{1.4}
\[	\begin{tabular}{ |c|c|c|c| } 
\hline 
Rep & $\dim^s$ & Borels & Stabilizer\\
\hline 
$K_{1|2}(t\epsilon_1)$, $t\neq 0,1$ & $(2|2)$ &$\epsilon\delta\delta$, $\delta\delta\epsilon$ &$\o\s\p(1|2)$  \\ 
\hline
$K_{1|2}(0)$ & $(2|2)$ &$\epsilon\delta\delta$, $\delta\delta\epsilon$ &$\C\times \o\s\p(1|2)$  \\ 
\hline
$K_{1|2}(\epsilon_1)$ & $(2|2)$ & $\delta\delta\epsilon$ & $\s\p(2)\ltimes \Pi Sp_2$ \\
\hline
$\Pi K_{1|2}(t\epsilon_1)$, $t\neq 0$ & $(2|2)$ & $\delta\epsilon\delta$ & $(\b^{-\delta\epsilon\delta})_{((1-t)\epsilon-\delta_2)}$ \\
\hline
\end{tabular} \label{gl(1,2)_table}
\]
\renewcommand{\arraystretch}{1}

\subsection{Some Spherical Irreducibles for $\g\l(m|n)$}

From our above work we may now assume either $m\geq 2$ or $n\geq 3$.

We present the (numerically) spherical modules for $\g\l(m|n)$ that arise naturally from the standard module, and prove they are spherical. We write $v_1,\dots,v_m,w_1,\dots,w_n$ for a homogeneous basis for $GL_{m|n}$.  Here $v_i$ is even with weight $\epsilon_i$ and $w_j$ is odd with weight $\delta_j$.

\begin{prop}\label{gl_std_reps}
	Suppose $1\leq m\leq n$ and either $m\geq 2$ or $n\geq 3$.  Then of all the modules $(\Pi)S^dGL_{m|n}$ and $(\Pi)\Lambda^dGL_{m|n}$, where $d\geq 1$, the numerically spherical ones are exactly those in the following list, and they are all spherical:
	\[
(\Pi)GL_{m|n}, \ \ \ S^2GL_{m|n}, \ \ \ \Lambda^2GL_{m|n}, \ \ \ \Pi S^2GL_{n|n}\cong(\Pi\Lambda^2GL_{n|n})^{\Pi}, \ \ \ \Pi S^2GL_{n|n+1}
	\]  
	Further, these modules are equivalent to one showing up in the following table, where we allow for $m,n$ to be arbitrary.  The table also lists the conjugacy classes of Borels each is spherical with respect to along with the stabilizer of a spherical vector $v$.
\end{prop}
\renewcommand{\arraystretch}{1.5}
	\[\hspace{-1em}	\begin{tabular}{ |c|c|c|c| } 
	\hline 
	Rep & $\dim^s$ & Borels & Stabilizer\\
	\hline
	$GL_{m|n}$ & $(m|n)$ & $\cdots\epsilon$ & $\g\l(m-1|n)\ltimes(\C^{m-1|n})^*$ \\
	\hline
	$S^2GL_{m|n}$ & $(\frac{n(n-1)}{2}+\frac{m(m+1)}{2}|mn)$& $\delta^{i_1}\epsilon^{j_1}\delta^{2i_2}\epsilon^{j_2}\delta^{2i_3}\cdots\epsilon^{j_k}\delta^{2i_l}$& $\o\s\p(GL_{m|n}^*,v)$ \\ 
	\hline
	\makecell{$(\Pi \Lambda^2GL_{n|n})^{\Pi}\cong$\\ $\Pi S^2GL_{n|n}$}  & $(n^2|n^2)$& $\epsilon\delta\epsilon\delta\cdots\epsilon\delta$ & $\p(n)^\Pi$\\
	\hline
	$\Pi S^2GL_{n|n+1}$ & $(n(n+1)|n(n+1))$& $\delta\epsilon\delta\epsilon\cdots\epsilon\delta$ & $(\p(n)^\Pi\times\C)\ltimes (\Pi P_{n|n})_{-\delta_1}$ \\
	\hline
	\end{tabular}
	\]
	\renewcommand{\arraystretch}{1}
	By $\o\s\p(GL_{m|n}^*,v)$ we mean the superalgebra of matrices preserving the form on $GL_{m|n}^*$ induced by the spherical vector $v$.  By $\p(n)^\Pi$, we mean the algebra gotten by applying the parity shift automorphism to $\p(n)$.  In the last stabilizer, the extra copy of $\C$ is acting by the character $-1$ (this extra copy of $\C$ is spanned by the diagonal element dual to the weight $\delta_{1}$, hence the notation).

\begin{proof} Observe that $GL_{m|n}\cong\Pi GL_{n|m}$, and so $\Lambda^dGL_{m|n}\cong\Pi^dS^dGL_{n|m}$ for all $d$.  It therefore suffices to study $GL_{m|n}$ and $(\Pi)S^dGL_{m|n}$ for $d\geq 2$, where we now only require that $m,n\geq 1$ and that either $m=n=2$ or $\max(m,n)\geq 3$.  We have the following cases; let $V=GL_{m|n}$.
	
\begin{enumerate}
			
\item $V$: We observe that $V$ is spherical if we take $v_m$ for our spherical vector, exactly with respect to Borels for which $\delta_i-\epsilon_m$ is positive for all $i$.
	
\item \emph{Excluding $S^kV$ and $\Pi S^kV$ for $k>2$}: If $S^kV$ is spherical then $S^kV_{\ol{0}}$ is $\g_{0}$-spherical, which implies $m=1$- but if $m=1$ the odd part of $S^kV$ will be too large.  If $\Pi S^kV$ is spherical, then $S^{k-1}V_{\ol{0}}\otimes V_{\ol{1}}$ is $\g_0$-spherical, which implies either $m=1$ or $n=1$.  In both of these cases, however, the odd part of $\Pi S^kV$ will be too large.  
 
We are left to look at $(\Pi)S^2V$.
	
\item \emph{$S^2V$ is always spherical:}  A homogeneous basis for $S^2V$ is 
\[
\text{Even}: \ v_iv_j \ \ i\leq j,\ \  w_iw_j \ \ i<j;\ \ \text{Odd}:\ \  v_iw_j
\]  
	
The vector $v=v_1^2+\dots+v_m^2+w_1w_2+\cdots+w_{n-1}w_n$ is a $\g_0$-spherical vector, so by \cref{cspher_vec}, $S^2V$ is spherical with respect to a Borel $\b$ defined by an $\epsilon\delta$ string if and only if $\b\cdot v=V$.  
	
Now a Borel $\b$ defined by an $\epsilon\delta$ string will contain the odd operators $w_i\d_{v_j}$ exactly when $\delta_i-\epsilon_j$ is positive, and will have odd operators $v_i\d_{w_j}$ exactly when $\epsilon_i-\delta_j$ is positive.  Further, we see that $\frac{1}{2}w_i\d_{v_j}(v)=w_iv_j$ and $v_i\d_{w_j}(v)=v_i(w_{j-1}+w_{j+1})$, where we let $w_{0}=w_{n+1}=0$.  

Now fix an $i$ with $1\leq i\leq m$.  Because of our choice of even Borel, we observe that if $\epsilon_i-\delta_j$ is positive, then so is $\epsilon_i-\delta_k$ for $k>j$.  Also, if $\delta_j-\epsilon_i$ is positive, so is $\delta_k-\epsilon_i$ for $k<j$.  Hence we may choose $j$ so that for a $k\leq j$, $\delta_k-\epsilon_i$ is positive, and for $j<k$, $\epsilon_i-\delta_k$ is positive.  Then in $\b_{\ol{1}}\cdot v$ we will get all monomials $v_iw_k$ for all $k\leq j$ along with all terms $v_i(w_{k-1}+w_{k+1})$ for $j<k\leq n$.  It is now a linear algebra exercise to show that these monomials span the subspace spanned by $\{v_iw_k:1\leq k\leq n\}$ if and only if $n-j$ is even.  Therefore, $\b_{\ol{1}}\cdot v=(S^2V)_{\ol{1}}$ if and only if the number of $\delta$'s appearing after any $\epsilon$ in the $\epsilon\delta$-string is even- i.e. the $\epsilon\delta$ string must be of the form 
\[
\delta^{i_1}\epsilon^{j_1}\delta^{2i_2}\epsilon^{j_2}\delta^{2i_3}\cdots\epsilon^{j_k}\delta^{2i_l}
\]

\item \emph{Examination of $\Pi S^2V$:} The even part is $V_{\ol{0}}\otimes V_{\ol{1}}$, and its odd part is $S^2V_{\ol{0}}\oplus \Lambda^2V_{\ol{1}}$.  Therefore the dimension of the odd part is $\frac{n^2+m^2+m-n}{2}$.  This is less than or equal to $nm$ if and only if  $0\leq n-m\leq 1$.  This leaves the cases $V=GL_{n|n}$ and $V=GL_{n|n+1}$.  We show these modules are indeed spherical.  
	
	For the case of $V=GL_{n|n}$, a $\g_0$-spherical vector is given by $v=v_1w_1+\dots+v_mw_m$.  Let $\b$ be a Borel defined by an $\epsilon\delta$ string.  Then $\Pi S^2V$ is spherical for a Borel $\b$ if and only if $\b_{\overline{1}}\cdot v=(\Pi S^2V)_{\ol{1}}$.  The Borel has odd operators exactly as described in the case for $S^2V$.  Here we see that $v_i\d_{w_j}(v)=v_iv_j$  and $w_i\d_{v_j}(v)=w_iw_j$.  Now in order to get $v_1^2$, we must have $\epsilon_1-\delta_1$ is positive, so the $\epsilon\delta$ string must start with $\epsilon$.  In order to get $w_1w_2$, we must have $\delta_1-\epsilon_2$ is positive, so the $\epsilon\delta$ string must start with $\epsilon\delta$.  Similarly, to get $v_2^2$ we must have $\epsilon_2-\delta_2$ positive, and so on, so that continuing this way we find that the Borel must have $\epsilon\delta$ string $\epsilon\delta\epsilon\delta\cdots\epsilon\delta$.  Further, the representation is spherical with respect to this Borel.
	
	For $V=GL_{n|n+1}$, a $\g_0$-spherical vector is given by $v=v_1w_2+v_2w_3+\dots+v_mw_{m+1}$.  Let $\b$ be a Borel defined by an $\epsilon\delta$ string.  We see that $v_i\d_{w_j}(v)=v_iv_{j-1}$, and $w_i\d_{v_j}(v)=w_iw_{j+1}$.  Following the same idea as in the previous case, in order to get $w_1w_2$ we must have $\delta_1-\epsilon_1$ positive, and to get $v_1^2$ we must have $\epsilon_1-\delta_2$ positive, and continuing on like this we find that our Borel must have $\epsilon\delta$ string $\delta\epsilon\delta\epsilon\cdots\epsilon\delta$.  Further, this representation is spherical with respect to this Borel.
\end{enumerate}
\end{proof}
	
\begin{remark}
	The representation $S^2GL_{m|2n}$ and its relation to the Capelli problem for the symmetric pair $(\g\l(m|2n),\o\s\p(m|2n))$ is studied in \cite{sahi2016capelli}.
\end{remark}

\subsection{$\g\l(1|n)$}

\textbf{Candidate even weights:}
\[
t\epsilon_1\ (t\neq0),\ \ \ \ -\epsilon_1+\delta_1,\ \ \ \ -\delta_n  
\]
\[
-2\epsilon_1+2\delta_1, \ \ \ \ -2\delta_n,\ \ \ \ -\epsilon_1+\delta_1+\delta_2,\ \ \ \ -\delta_{n-1}-\delta_n  
\]
where $t\in\C$.  

\

\textbf{Check for (numerical) sphericity of $L(\lambda)$:} In the following table we go through each possible case, grouping them appropriately according to how we either prove they do not give a (numerically) spherical representation or explain why they are covered by \cref{gl_std_reps}.  The technique is the exact same as what was used when studying $\g\l(1|2)$.

\renewcommand{\arraystretch}{1.4}
\[
\begin{tabular}{|c|c|c|}
\hline
$\lambda$ & Action & Conclusion\\
\hline
$t\epsilon_1$, $t\notin\Z_{>0}$ & \makecell{ Weight is typical; apply\\ three odd reflections, get two\\ odd highest weights} & Odd part too large\\
\hline
$t\epsilon_1$, $t\in\Z_{>0}$ & \multirow{6}{*}{\makecell{After adding multiple of Berezinian\\ get something of the form\\ $(\Pi)S^kGL_{1|n}^{(*)}$ or $(\Pi)\Lambda^kGL_{1|n}^{(*)}$}} & \multirow{6}{*}{\makecell{Falls under cases considered\\ by \cref{gl_std_reps}}}\\
$-\epsilon_1+\delta_1$ &&\\
$-\delta_n$ &&\\
$-2\delta_n$&&\\
$-\epsilon_1+\delta_1+\delta_2$ &&\\
$-\delta_{n-1}-\delta_n$&&\\
\hline 
$-2\epsilon_1+2\delta_1$ & \makecell{Apply $r_{\epsilon_1-\delta_2}\circ r_{\epsilon_1-\delta_1}$ get odd\\ highest weight $-3\epsilon_1+2\delta_1+\delta_2$} & Odd part too large\\
\hline
\end{tabular}
\]
\renewcommand{\arraystretch}{1}

\

\textbf{Candidate odd weights:}
\[
t\epsilon_1 \ (t\neq0),\ \ \ t\epsilon_1+\delta_1,\ \ \ -\delta_n
\]
\textbf{Check for (numerical) sphericity of $\Pi L(\lambda)$:}
\renewcommand{\arraystretch}{1.5}
\[
\begin{tabular}{|c|c|c|}
\hline
$\lambda$ & Action & Conclusion\\
\hline 
$t\epsilon_1$, $t\in\Z_{>0}$ & \multirow{3}{*}{\makecell{After adding multiple of Berezinian\\ get something of the form\\ $(\Pi)S^kGL_{1|n}^{(*)}$ or $(\Pi)\Lambda^kGL_{1|n}^{(*)}$}} & \multirow{3}{*}{\makecell{Falls under cases considered\\ by \cref{gl_std_reps}}}\\
$t\epsilon_1+\delta_1$, $t=\pm1$ & &\\
$-\delta_n$ & &\\
\hline
$t\epsilon_1$, $t\notin\Z_{> 0}$ & \multirow{2}{*}{\makecell{Apply $r_{\epsilon_1-\delta_2}\circ r_{\epsilon_1-\delta_1}$; get\\ new odd highest weight}} & \multirow{2}{*}{Odd part too large}\\
$t\epsilon_1+\delta_1$, $t\neq\pm1$ & &\\
\hline 
\end{tabular}
\]
\renewcommand{\arraystretch}{1}

\textbf{(Numerically) Spherical irreducibles for $\g\l(1|n)$, $n\geq 3$:} Along with $(\Pi)\C$, we have the following numerically spherical irreducibles:
\[
(\Pi)GL_{1|n},\ \ \ \ S^2GL_{1|n},\ \ \ \ \Lambda^2GL_{1|n}
\]

\textbf{Indecomposable Spherical Modules:}  The only extensions which could be spherical in this case are extensions of the trivial even module by a numerically spherical module, or an extension of $\Pi GL_{1|n}$ by a numerically spherical module which is geometrically equivalent to either $\Pi\C$ or $\Pi GL_{1|n}$.

In this case, the trivial even module has non-trivial extensions only with $\Pi^n(S^{n-1}V^*)_{-\Ber}$ and $\Pi^n(S^{n-1}V)_{\Ber}$.  These modules are never numerically spherical for $n\geq 3$.

The module $\Pi GL_{1|n}$ has non-trivial extensions only with $\Pi^{n-1}(S^{n-2}V^*)_{-\Ber}$ and $L(n\epsilon_1-\delta_2-\dots-\delta_n)$ (up to a parity shift).  The latter module is not numerically spherical.  Although the former is numerically spherical, its extensions with $\Pi GL_{1|n}$ will never be numerically spherical.  It follows that we get no new spherical indecomposables.

\subsection{The General Case $\g\l(m|n)$, $m\geq 2$}

\textbf{Candidate even weights:}
\[
t{\det}_\epsilon\  (t\neq0),\ \ \epsilon_1,\ \ {\det}_\epsilon-\epsilon_m,\ \ 2\epsilon_1, \ \ 2{\det}_\epsilon-2\epsilon_m \ (m>2) 
\]
\[
\epsilon_1+\epsilon_2 \ (m>2),\ \ {\det}_\epsilon-\epsilon_{m-1}-\epsilon_m  \ (m>2)
\]
\[
-{\det}_\delta+\delta_1,\ \ -\delta_n,\ \ -2{\det}_\delta+2\delta_1 \ (n>2)\ \ -2\delta_n 
\]
\[
-{\det}_\delta+\delta_1+\delta_2 \ (n>2),\ \ -\delta_{n-1}-\delta_{n} \ (n>2)
\]	
\[
-{\det}_\epsilon+\epsilon_1+\delta_1,\ \ \epsilon_1-\delta_n, \ \ -\epsilon_m+\delta_1,\ \ {\det}_\epsilon-\epsilon_m-\delta_n
\]
\textbf{Check for (numerical) sphericity of $L(\lambda)$:}

\renewcommand{\arraystretch}{1.3}
\[
\begin{tabular}{|c|c|c|}
\hline
$\lambda$ & Action & Conclusion\\
\hline
$t\det_{\epsilon}$, $t=0,\pm1$ & \multirow{11}{*}{\makecell{After adding multiple\\ of Berezinian\\ $L(\lambda)$ becomes\\ something of the form\\ $(\Pi)S^kGL_{m|n}^{(*)}$ or \\$(\Pi)\Lambda^kGL_{m|n}^{(*)}$}} & \multirow{11}{*}{\makecell{Falls under cases considered\\ by \cref{gl_std_reps}}}\\
$\epsilon_1$&&\\
$\det_{\epsilon}-\epsilon_m$ &&\\
$2\epsilon_1$ &&\\
$\epsilon_1+\epsilon_2$, $m>2$&&\\
$\det_{\epsilon}-\epsilon_{m-1}-\epsilon_m$, $m>2$,  &&\\
$-\det_{\delta}+\delta_1$&&\\
$-\delta_n$&&\\
$-2\delta_n$&&\\
$-\det_{\delta}+\delta_1+\delta_2$, $n>2$ &&\\
$-\delta_{n-1}-\delta_n$, $n>2$&&\\
\hline
$2\det_{\epsilon}-2\epsilon_m$, $m>2$ & \makecell{Apply $r_{\epsilon_{m-1}-\delta_1}\circ r_{\epsilon_{m}-\delta_1}$\\ get odd \\highest weight\\ $2\epsilon_1+\dots+2\epsilon_{m-2}+$\\ $\epsilon_{m-1}+\delta_1$} & Odd part too large\\
\hline  
$-2\det_{\delta}+2\delta_1$, $n>2$ & \makecell{Apply $r_{\epsilon_{m}-\delta_2}\circ r_{\epsilon_{m}-\delta_1}$\\ get odd \\highest weight\\ $-\epsilon_m-\delta_2-$\\$2\delta_3-\dots-2\delta_n$} & Odd part too large\\
\hline
\end{tabular}
\]
\renewcommand{\arraystretch}{1}
The final cases to consider for $\lambda$ are:
\[
-{\det}_{\epsilon}+\epsilon_1+\delta_1,\ \ \epsilon_1-\delta_n,\ \ -\epsilon_m+\delta_1,\ \ {\det}_{\epsilon}-\epsilon_m-\delta_n
\]
\begin{itemize}
	\item[$\bullet$] If $\lambda=-\det_{\epsilon}+\epsilon_1+\delta_1$, then applying $r_{\epsilon_m-\delta_2}\circ r_{\epsilon_m-\delta_1}$ we get odd highest weight $-\det_{\epsilon}+\epsilon_1-\epsilon_m+\delta_1+\delta_2$.  If $n>2$, then this shows the odd part is too large.  
	
	If $n=m=2$, then $\lambda=-\epsilon_2+\delta_1$.  Applying $r_{\epsilon_2-\delta_2}\circ r_{\epsilon_2-\delta_1}$ gives odd highest weight $-2\epsilon_2+\delta_1+\delta_2$, generating a $3$-dimensional submodule of the odd part.  On the other hand, if we instead applied $r_{\epsilon_1-\delta_1}\circ r_{\epsilon_2-\delta_1}$, we'd get odd highest weight $-\epsilon_1-\epsilon_2+2\delta_1$, giving a distinct $3$-dimension submodule of the odd part.  Therefore, the odd part has to be at least 6-dimensional, which is too large.
		
	\item[$\bullet$]  If $\lambda=\epsilon_1-\delta_n$, this is the parity shift of the adjoint module, which has too large an odd part.
		
	\item[$\bullet$] If $\lambda=-\epsilon_m+\delta_1$, then applying $r_{\epsilon_m-\delta_1}$ followed by either $r_{\epsilon_m-\delta_2}$ or $r_{\epsilon_{m-1}-\delta_1}$ gives odd highest weights $-\epsilon_{m-1}-\epsilon_m+2\delta_1$ and $-2\epsilon_m+\delta_1+\delta_2$.  This shows the odd part will be too large.  
		
	\item[$\bullet$] Finally, if $\lambda=\det_{\epsilon}-\epsilon_m-\delta_n$, then applying $r_{\epsilon_{m-1}-\delta_1}\circ r_{\epsilon_m-\delta_1}$ gives odd highest weight $\epsilon_1+\dots+\epsilon_{m-2}+\delta_1-\delta_n$, which is of dimension $n^2-1+m(m-1)/2$, which is bigger than $nm$ whenever $n>2$, $m\geq 2$. 
	
	Therefore we can assume $n=m=2$, in which case $n^2-1+m(m-1)/2=4$.  Then we can also apply $r_{\epsilon_2-\delta_2}\circ r_{\epsilon_2-\delta_1}$ to get distinct odd highest weight $\epsilon_1-\epsilon_2$, making the odd part too large.
\end{itemize}

\textbf{Candidate odd weights:}
\[
t{\det}_{\epsilon}\  (t\neq0),\ \ \ \epsilon_1, \ \ \ \epsilon_1+\epsilon_2 \ (m\geq 3), \ \ \ 2\epsilon_1
\]
\[
-\delta_n, \ \ \ -\delta_{n-1}-\delta_n \ (n\geq 3), \ \ \ -2\delta_n
\]
\[
\epsilon_1+\delta_1-{\det}_{\epsilon}, \ \ \ \epsilon_1-\delta_n, \ \ -\epsilon_m+\delta_1,\ \ \ -\epsilon_m-\delta_n+{\det}_{\epsilon}
\]
\textbf{Check for (numerical) sphericity of $\Pi L(\lambda)$:}

\renewcommand{\arraystretch}{1.3}
\[
\begin{tabular}{|c|c|c|}
\hline
$\lambda$ & Action & Conclusion\\
\hline
$t\det_{\epsilon}$, $t=0,\pm1$ &  \multirow{7}{*}{\makecell{After adding multiple\\ of Berezinian\\ $\Pi L(\lambda)$ becomes\\ something of the form\\ $(\Pi)S^nGL_{m|n}^{(*)}$ or \\$(\Pi)\Lambda^nGL_{m|n}^{(*)}$}} & \multirow{7}{*}{\makecell{Falls under cases considered\\ by \cref{gl_std_reps}}}\\
$\epsilon_1+t\det_{\epsilon}$, $t=0$ &&\\
$2\epsilon_1+t\det_{\epsilon}$, $t=0$ &&\\
$\epsilon_1+\epsilon_2+t\det_{\epsilon}$, $m>2$, $t=0$&&\\
$-\delta_{n}+t\det_{\delta}$, $t=0$&&\\
$-2\delta_{n}+t\det_{\delta}$, $t=0$&&\\
$-\delta_{n-1}-\delta_{n}+t\det_{\delta}$, $n>2$ $t=0$&&\\
\hline
\end{tabular}
\]
\renewcommand{\arraystretch}{1}
This leaves us to consider the cases for $\lambda$:
\[
t{\det}_{\epsilon} \ t\neq 0,\pm1,\ \ -{\det}_{\epsilon}+\epsilon_1+\delta_1,\ \ \epsilon_1-\delta_n,\ \ -\epsilon_m+\delta_1, \ \ -{\det}_{\epsilon}-\epsilon_m-\delta_n
\]
\begin{itemize}
	\item[$\bullet$]  If $\lambda=t{\det}_{\epsilon} \ t\neq 0,\pm1$, then applying $r_{\epsilon_m-\delta_1}$ we get even highest weight $t\det_{\epsilon}-\epsilon_m+\delta_1$.  Next we may apply either $r_{\epsilon_{m-1}-\delta_1}$ or $r_{\epsilon_m-\delta_2}$, giving odd highest weights $t\det_{\epsilon}-\epsilon_{m-1}-\epsilon_m+2\delta_1$, or $t\det_{\epsilon}-2\epsilon_m+\delta_1+\delta_2$.  It follows the odd part will be too large.
	
	\item[$\bullet$] If $\lambda=-{\det}_{\epsilon}+\epsilon_1+\delta_1$ and $m=n=2$, then $\lambda$ differs by a multiple of the Berezinian with $\epsilon_1-\delta_2$, giving the adjoint representation which is not numerically spherical.  If $n>2$, then applying $r_{\epsilon_{m}-\delta_2}\circ r_{\epsilon_{m}-\delta_1}$ gives an even weight which is not $\g_{\ol{0}}$-spherical.  
	
	\item If $\lambda=\epsilon_1-\delta_n$ we get the adjoint module, which is not numerically spherical.
	
	\item If $\lambda=-\epsilon_m+\delta_1$, then if $m=n=2$ we again get a twist of the adjoint module, while if $n>2$, applying $r_{\epsilon_m-\delta_2}\circ r_{\epsilon_m-\delta_1}$ gives an even weight which is not $\g_{\ol{0}}$-spherical.
	
	\item Finally, if $\lambda=\det_{\epsilon}-\epsilon_m-\delta_n$ then when $n=m=2$ we get the adjoint module, while if $n>2$ then applying $r_{\epsilon_{m-1}-\delta_1}\circ r_{\epsilon_m-\delta_1}$ we get an even weight which is not $\g_{\ol{0}}$-spherical. 	  
 	
\end{itemize}

\textbf{Spherical irreducibles for $\g\l(m|n)$, $2\leq m\leq n$:} Apart from $(\Pi)\C$, numerically spherical irreducibles are all spherical in this case.
\[
GL_{m|n}, \ \ \Pi GL_{m|n},\ \  S^2GL_{m|n},\ \ \Lambda^2GL_{m|n},\ \ \Pi S^2GL_{n|n}, \ \ \Pi S^2GL_{n|n+1}
\]
We note however that if we remove the condition that $m\leq n$, then $\Pi GL_{m|n}$ is equivalent to $GL_{n|m}$ and $\Lambda^2GL_{m|n}$ is equivalent to $S^2GL_{n|m}$.

\

\textbf{Spherical indecomposables for $\g\l(m|n)$, $2\leq m\leq n$:}  In this case,  $S^2GL_{m|n}$, $\Lambda^2GL_{m|n},\Pi S^2GL_{m|m}$, and $\Pi S^2GL_{m|m+1}$ all have odd dimension equal to the odd dimension of a Borel.  Hence the only possible spherical extensions these modules could have are by a one-dimensional even module.  However there are no such extensions.  In fact, in this case the trivial module has no non-trivial extensions by any (numerically) spherical irreducibles, and nor does its parity shift.

This leaves us to look at the standard module.  However neither it nor its parity shift admits non-trivial extensions by numerically spherical modules.  It follows that all spherical indecomposable $\g\l(m|n)$-modules for $2\leq m\leq n$ are all irreducible.

\section{$\o\s\p(m|2n)$ Case}\label{osp_case}

We now study the case $\g=\o\s\p(m|2n)$, with  $m,n>0$, $(m,n)\neq (2,1)$ (since $\o\s\p(2|2)\cong\s\l(1|2)$).  For more about $\o\s\p(m|2n)$ and a description of $\b^{st}$-dominant weights, see \cite{cheng2012dualities} and \cite{musson2012lie}.   We refer the reader to appendix \hyperref[App_borels]{A} for a discussion of our notation for Borels and root systems.
\subsection{Modules From the Standard Representation}

Write $OSP_{m|2n}$ for the standard representation of $\o\s\p(m|2n)$ on $\C^{m|2n}$.
\begin{prop}\label{osp_std} \
	
	\begin{enumerate}
		\item  $OSP_{m|2n}$ is spherical if and only if $m\geq 2$.  In this case it is spherical exactly with respect to Borels with $\epsilon\delta$ strings of the form $\epsilon\cdots$, or $(\pm\epsilon)\delta^n$ when $m=2$.
		\item  $\Pi OSP_{m|2n}$ is always spherical exactly with respect to Borels with $\epsilon\delta$ strings of the form $\delta\cdots$.
	\end{enumerate} 
\end{prop}
\begin{proof}
For the first statement, $OSP_{1|2n}$ is not spherical because the odd dimension of any Borel of $\o\s\p(1|2n)$ is $n$, which is too small.  The rest is straightforward.
\end{proof} 
	
\

\subsection{$\o\s\p(2|2n)$, $n\geq 2$}
\textbf{Standard Borel:} $\b^{st}=\b^{\epsilon\delta^n}$.  The odd dimension of every Borel is $2n$.

\

\textbf{Spherical weights for $\g_{\ol{0}}=\o(2)\times\s\p(2n)$}:
\begin{align}\label{osp22n_wts}
\delta_1+s\epsilon_1, \ \ \ s\epsilon_1 \ (\nu\neq0), \ \ \ \delta_1+\delta_2+s\epsilon_1 \ (n=2)
\end{align}

\

\textbf{Candidate even weights} 
\[
\delta_1-\epsilon_1,\ \ \  s\epsilon_1\  (s\neq0),\ \ \ \delta_1+\delta_2-\epsilon_1 (n=2)
\]

\

\textbf{Candidate odd weights}
\[
s\epsilon_1\ (s\neq 0), \ \ \ \ -\epsilon_1+\delta_1
\]

\

\textbf{Check for (numerical) sphericity of $(\Pi)L(\lambda)$:}
\renewcommand{\arraystretch}{1.3}
\[
\begin{tabular}{|c|c|c|c|}
\hline
$\lambda$ & Parity & Action & Conclusion\\
\hline
$\delta_1-\epsilon_1$ & even & \multirow{3}{*}{\makecell{Apply $r_{\epsilon_1-\delta_2}\circ r_{\epsilon_1-\delta_1}$; get\\ new even or odd highest weight}} & \multirow{3}{*}{\makecell{Odd part too large\\ or even part not spherical}}\\
$s\epsilon_1$, $n>2$, $s\neq0,1$ & even & &\\
$s\epsilon_1$, $s\neq0,1$ & odd & &\\
\hline 
$s\epsilon_1$, $s=1$ & even & \multirow{3}{*}{Do nothing} & \multirow{3}{*}{\makecell{Isomorphic to standard module \\ up to parity shift. \\ Covered by \cref{osp_std}.}}\\
&&&\\
$s\epsilon_1$, $s=1$ & odd &&\\
\hline
$s\epsilon_1$, $n=2$, $s=3$& even & \multirow{3}{*}{\makecell{Do nothing}} & \multirow{3}{*}{\makecell{Isomorphic to $L(\delta_1+\delta_2-\epsilon)$ \\ or its dual- a spherical module}}\\
\makecell{$\delta_1+\delta_2-\epsilon_1$,\\ $n=2$} & even &&\\
\hline
\makecell{$s\epsilon_1$,\\$n=2$, $s\neq 1,3$} & even & \multirow{4}{*}{\makecell{Apply the odd reflections\\ $r_{\epsilon_1-\delta_1}$, $r_{\epsilon_1-\delta_2}$, $r_{\epsilon_1+\delta_2}$\\ Gives two odd highest weights}} & \multirow{4}{*}{Odd part too large.}\\
\makecell{$\delta_1-\epsilon_1$, \\$n=2$}& odd &&\\
\hline
\makecell{$\delta_1-\epsilon_1$,\\$n>2$} & odd & \makecell{Apply $r_{\epsilon_1-\delta_3}\circ r_{\epsilon_1-\delta_2}\circ r_{\epsilon_1-\delta_1}$\\ Get even highest weight vector} & Even part not spherical\\
\hline 
\end{tabular}
\]
\renewcommand{\arraystretch}{1}

\textbf{(Numerically) Spherical irreducibles for $\o\s\p(2|2n)$, $n\geq 2$}:  Apart from $\Pi\C$, all numerically spherical irreducibles are spherical.	
\[
OSP_{2|2n},\ \ \ \Pi OSP_{2|2n},\ \ \ L_{\b^{st}}(\delta_1+\delta_2-\epsilon_1)
\]

\

\subsection{$\o\s\p(2m|2n)$, $m\geq 2$}
\textbf{Standard Borel} $\b^{st}=\b^{\delta^n\epsilon^m}$.

\

\textbf{Spherical weights for $\g_{\ol{0}}=\o(2m)\times\s\p(2n)$:}
\renewcommand{\arraystretch}{1.4}
\[
\begin{tabular}{|c|c|}
\hline
$(m,n)$ & Weights\\
\hline
Any $m,n$ & $\delta_1, \epsilon_1$\\
\hline
Any $n$; $m=2,3,5$ & 	$\frac{1}{2}(\epsilon_1+\dots\pm\epsilon_m)$\\
\hline
Any $n$; $m=2,3$ & $\epsilon_1+\dots\pm\epsilon_m$\\
\hline
\makecell{Any $n$ and $m=2$,\\ or $n\leq 2$, $m=3$} & $\frac{1}{2}(\epsilon_1+\dots\pm\epsilon_m)+\delta_1$ \\
\hline
Any $m$; $n=2$ & $\delta_1+\delta_2$\\
\hline
Any $m$; $n=1$ & $2\delta_1$\\
\hline
\end{tabular}
\]
\renewcommand{\arraystretch}{1}

\textbf{Candidate even weights:} Only $\lambda=\delta_1$.  Although $\delta_1+\delta_2$ when $n=2$ and $2\delta_1$ when $n=1$ are dominant, applying a simple odd reflection gives an odd highest weight which is too large.

\

\textbf{Candidate odd weights:} Only $\lambda=\delta_1$.

\

\textbf{Check for (numerical) sphericity of $(\Pi)L(\lambda)$:}
\renewcommand{\arraystretch}{1.5}
\[
\begin{tabular}{|c|c|c|c|}
\hline
Weight & Parity & Action & Conclusion\\
\hline
$\delta_1$ & even & \multirow{2}{*}{Do nothing} & \multirow{2}{*}{\makecell{Parity shift of standard module;\\ covered by \cref{osp_std}}}\\
$\delta_1$ & odd &&\\
\hline
\end{tabular}
\]
\renewcommand{\arraystretch}{1}

\textbf{(Numerically) Spherical irreducibles for $\o\s\p(2m|2n)$, $m\geq 2$:}  Along with $(\Pi)\C$, we have two spherical irreducibles:
\[
OSP_{2m|2n},\ \ \ \Pi OSP_{2m|2n}
\]

\

\subsection{$\o\s\p(2m+1|2n)$}
 
\textbf{Standard Borel:} $\b^{st}=\b^{\delta^n\epsilon^m}$.

\

\textbf{Spherical weights for $\g_{\ol{0}}=\o(2m+1)\times\s\p(2n)$:}
\renewcommand{\arraystretch}{1.4}
\[
\begin{tabular}{|c|c|}
\hline
$(m,n)$ & Weights\\
\hline
Any $m,n$ & $\delta_1, \epsilon_1$\\
\hline
Any $n$; $m\leq 4$ & 	$\frac{1}{2}(\epsilon_1+\dots+\epsilon_m)$\\
\hline
\makecell{$(1,1)$, $(1,2)$,\\ or $(2,1)$} & $\frac{1}{2}(\epsilon_1+\dots+\epsilon_m)+\delta_1$ \\
\hline
Any $m$; $n=2$ & $\delta_1+\delta_2$\\
\hline
Any $m$; $n=1$ & $2\delta_1$\\
\hline
\end{tabular}
\]
\renewcommand{\arraystretch}{1}

\textbf{Candidate even weights:} Only $\delta_1$.  Again, the other $\b^{st}$-dominant weights appearing the above table have too large an odd part as seen by applying the odd simple reflection $r_{\delta_n-\epsilon_1}$.

\

\textbf{Candidate odd weights:} Only $\lambda=\delta_1$.  

\

\textbf{Check for sphericity of $(\Pi)L(\lambda)$:}
\renewcommand{\arraystretch}{1.5}
\[
\begin{tabular}{|c|c|c|c|}
\hline
Weight & Parity & Action & Conclusion\\
\hline
$\delta_1$ & even & \multirow{2}{*}{Do nothing} & \multirow{2}{*}{\makecell{Parity shift of standard module;\\ covered by \cref{osp_std}}}\\
$\delta_1$ & odd &&\\
\hline
\end{tabular}
\]
\renewcommand{\arraystretch}{1}

\textbf{Spherical indecomposables for $\o\s\p(m|2n)$, $(m,n)\neq (2,1)$:}  The trivial module has a distinct central character from $OSP_{m|2n}$ for all $(m,n)$, and therefore there are no new spherical indecomposables for $\o\s\p(m|2n)$ when $(m,n)\neq(2,2)$.

When $m=n=2$, the central character of $L_{\b^{st}}(\delta_1+\delta_2-\epsilon_1)$ is the same as that of $OSP_{2|4}$, and therefore it has no extensions with a trivial module.  Any extension of $L_{\b^{st}}(\delta_1+\delta_2-\epsilon_1)$ with $OSP_{2|4}$  will have too large an odd part, and therefore cannot be spherical.  So $\o\s\p(2|2n)$ also has no new spherical indecomposables.

Below is a table of all spherical irreducibles for $\o\s\p(m|2n)$, as well as Borels for which sphericity is achieved, and stabilizers of spherical vectors:

\renewcommand{\arraystretch}{1.5}
\[\hspace{-1em}	\begin{tabular}{ |c|c|c|c|c| } 
\hline 
Rep &  $\dim^s$ & Borels & Stabilizer\\
\hline
$OSP_{m|2n}$, $m\neq 1$ & $(m|2n)$ & $\epsilon\cdots$ & $\o\s\p(m-1|2n)$ \\
\hline 
$\Pi OSP_{m|2n}$ & $(2n|m)$ & $\delta\cdots$ & $\o\s\p(m-1|2n-2)\ltimes(\C\oplus \C^{m-1|2n-2})$\\
\hline
$L(\delta_1+\delta_2-\epsilon_1)$, $m=n=2$ &  $(6|4)$ & $(\pm\epsilon)\delta\delta,\delta\delta\epsilon$ & $\o\s\p(1|2)\times\o\s\p(1|2)$ \\ 
\hline
\end{tabular}
\]
\renewcommand{\arraystretch}{1}

\section{Exceptional Basic Simple Algebras}\label{ex_case}

Here we consider when $\g=G(1,2),F(1,3),$ and $D(1,2;a)$.  We show that none of these algebras have nontrivial spherical modules (unless $D(1,2;a)\cong\o\s\p(2|4)$).  We refer to L. Martirosyan's thesis \cite{martirosyan2014representation} for more on $G(1,2)$ and $F(1,3)$.  For the description of the root system of $D(1,2;a)$ used below, see \cite{serganova1996generalizations}, and for the parametrization of dominant weights we refer to example 10.7 of \cite{serganova2011kac}.

\subsection{$G(1,2)$ Case} 

\textbf{Root system:}  Take $\h^*$ to spanned by $\epsilon_1,\epsilon_2,\epsilon_3,\delta$ with the relation $\epsilon_1+\epsilon_2+\epsilon_3=0$, and inner product $(\epsilon_i,\epsilon_i)=-2(\epsilon_i,\epsilon_j)=-(\delta,\delta)=2$, where $i\neq j$.  Then the roots are given by:
\[
\Delta_{\ol{0}}=\{\pm\epsilon_i,\epsilon_i-\epsilon_j;\pm2\delta\},\ \ \ \Delta_{\ol{1}}=\{\pm\delta;\pm\delta\pm\epsilon_j\}.
\]

The following table goes through our usual to-do list, and shows there are no non-trivial numerically spherical irreducibles.

\[
\begin{tabular}{|c|c|}
\hline
\multirow{2}{*}{$\b^{st}$} & \multirow{2}{*}{Borel with simple roots $\delta+\epsilon_3,\epsilon_1,\epsilon_2-\epsilon_1$.}\\
&\\
\hline
\multirow{2}{*}{$\g_{\ol{0}}$-spherical weights} & \multirow{2}{*}{$\delta,\ \ 2\delta, \ \ \epsilon_1+\epsilon_2$}\\
&\\
\hline
Candidate even weights & \makecell{None; $2\delta$ is dominant, but applying\\ $r_{\delta+\epsilon_3}$ gives too large an odd part}\\
\hline
\multirow{2}{*}{Candidate odd weights} & \multirow{2}{*}{None}\\
&\\
\hline
\end{tabular}
\]

\subsection{$F(1,3)$ Case}

\textbf{Root system:}  We present $\h^*$ as the vector space with basis $\epsilon_1,\epsilon_2,\epsilon_3,\delta$, and inner product $(\epsilon_i,\epsilon_j)=\delta_{ij}$ and $(\delta,\delta)=-3$.  Then the roots of $\g$ are
\[
\Delta_{\ol{0}}=\{\pm\epsilon_i\pm\epsilon_j,\pm\epsilon_i;\pm\delta\} \ \ \ \Delta_{\ol{1}}=\{\frac{1}{2}(\pm\epsilon_1\pm\epsilon_2\pm\epsilon_3\pm\delta)\}
\]

Again the following table shows there are no non-trivial numerically spherical irreducibles.
\[
\begin{tabular}{|c|c|}
\hline
\multirow{3}{*}{$\b^{st}$} & \multirow{3}{*}{\makecell{Borel with simple roots\\ $\frac{1}{2}(-\epsilon_1-\epsilon_2-\epsilon_3+\delta),\epsilon_3,\epsilon_2-\epsilon_3,\epsilon_1-\epsilon_2$}}\\
&\\
&\\
\hline
\multirow{2}{*}{$\g_{\ol{0}}$-spherical weights} & \multirow{2}{*}{$\frac{1}{2}\delta,\ \ \delta,\ \ \epsilon_1, \ \ \frac{1}{2}(\epsilon_1+\epsilon_2+\epsilon_3)$}\\
&\\
\hline
Candidate even weights & \makecell{None; $\delta$ is dominant, but \\ applying $r_{\frac{1}{2}(-\epsilon_1-\epsilon_2-\epsilon_3+\delta)}$ gives \\ too large an odd part}\\
\hline
\multirow{2}{*}{Candidate odd weights} & \multirow{2}{*}{None}\\
&\\
\hline
\end{tabular}
\]

\subsection{$D(2,1;\alpha)$ Case}

\textbf{Root system:}  We may present $\h^*$ as $\C\langle \epsilon,\delta,\gamma\rangle$, with inner product
\[
(\epsilon,\delta)=(\epsilon,\gamma)=(\delta,\gamma)=0,\ \ \ (\epsilon,\epsilon)=\frac{1}{\alpha}(\delta,\delta)=-\frac{1}{1+\alpha}(\gamma,\gamma)=1
\]
The roots are:
\[
\Delta_{\ol{0}}=\{\pm\epsilon,\pm\delta,\pm\gamma\},\ \ \ \Delta_{\ol{1}}=\{\frac{1}{2}(\pm\epsilon\pm\delta\pm\gamma)\}
\]

\textbf{Standard Borel:} $\b^{st}$ is a Borel with three simple isotropic roots $\alpha_1,\alpha_2,\alpha_3$ (such a Borel is unique up to conjugacy).  Then the principal roots are $\beta_i=\alpha_j+\alpha_k$, where $\{i,j,k\}=\{1,2,3\}$, i.e. they are even roots which are simple after application of some number of odd reflections to the simple root system $\b^{st}$.

\

\textbf{$\b^{st}$-dominant weights:} Since the principal roots are a basis of the $\h_{\ol{0}}^*$, we may parametrize a weight $\lambda$ by $(c_1,c_2,c_3)$, where $c_i=\lambda(h_{\beta_i})$.  The conditions that $\lambda$ is dominant integral with respect to this Borel is that $c_1,c_2,c_3\in\Z_{\geq0}$, and one of the following holds:
\begin{enumerate}
	\item $c_1,c_2,c_3\in\Z_{>0}$
	\item $c_1=(a+1)c_2+c_3=0$
	\item $c_2=-ac_1+c_3=0$
	\item $c_3=-ac_1+(a+1)c_2=0$
\end{enumerate}

\

\textbf{Spherical weights for $\g_{\ol{0}}$:}  
\[
(x,0,0),\ \ (0,x,0),\ \ (0,0, x)\ \text{ where }x=1,2
\]
and
\[
(a,b,c)\ \text{ where two out of }a,b,c\text{ are 1, and the other is 0.}
\]

\

\textbf{Candidate even weights:}  
\[
(0,1,1) \text{ with } \alpha=-2, \ \ \ (1,0,1) \text{ with } \alpha=1
\]

\

\textbf{Check sphericity of $L(\lambda)$:}  We have $D(1,2;1)\cong D(1,2;-2)\cong\o\s\p(2|4)$, so these fall under \cref{osp_case}.

\

\textbf{Candidate odd weights:} 
\[
(0,1,1) \text{ with } \alpha=-2, \ \ \ (1,0,1) \text{ with } \alpha=1
\]

\

\textbf{Check sphericity of $\Pi L(\lambda)$:}  The cases are again covered by \cref{osp_case}.

\

\textbf{(Numerically) Spherical irreducibles for $D(1,2;\alpha)$:} None, unless $D(1,2;\alpha)\cong\o\s\p(2|4)$.  Therefore there are no new spherical indecomposable modules.

\section{The Case $\p(n)$}\label{p_case}

Let $\g=\p(n)$, $n\geq 2$.  We refer the reader to \cite{balagovic2016translation}, in particular for computations of dual representations, as well as \cite{chen2015finite} and \cite{serganova2002representations} for more on the representation theory of this algebra.

\textbf{Notation:}
A matrix presentation for $\p(n)$, under the representation of the standard module, is 
\begin{align}\label{p_pres}
\begin{bmatrix}A & B\\C & -A^{t}\end{bmatrix}
\end{align}
where $B^t=B$, $C^t=-C$.

We have a $\Z$-grading $\g=\g_{-1}\oplus\g_0\oplus\g_{1}$, where $\g_{0}=\g_{\ol{0}}$, $\g_{1}$ are matrices with $A=C=0$ and $\g_{-1}$ are matrices with $A=B=0$.  

Write $P_{n|n}$ for the standard module of $\p(n)$, and $q\in (S^2P_{n|n}^*)_{\ol{1}}$ for a non-degenerate odd form on $P_{n|n}$ preserved by $\p(n)$.  Then $q$ induces an isomorphism $P_{n|n}^*\cong\Pi P_{n|n}$.

\

\textbf{Root system:} Write $\h$ for the (even) Cartan subalgebra of diagonal matrices.  Let $\epsilon_1,\dots,\epsilon_n$ be the usual basis of $\h^*$.  Then the roots are:
\[
\Delta_{\ol{0}}=\{\epsilon_i-\epsilon_j\text{ for }i\neq j\}\ \ \ \Delta_{\ol{1}}=\{\epsilon_i+\epsilon_j\text{ for }i\leq j\}\sqcup\{-(\epsilon_i+\epsilon_j)\text{ for }i<j\}
\]

\

\textbf{Standard Borel:} $\b^{st}=\b_{\ol{0}}^{st}\oplus\g_{-1}$, where $\b_{\ol{0}}^{st}$ is the Borel of $\g_{\ol{0}}$ with simple roots $\epsilon_1-\epsilon_2,\dots,\epsilon_{n-1}-\epsilon_n$.     

%\
%
%\textbf{$\b^{st}$-dominant weights:}
%\[
%\lambda=\lambda_1\epsilon_1+\dots+\lambda_n\epsilon_n, \  \lambda_i\in\C, \ \lambda_i-\lambda_{i+1}\in\Z_{\geq0}\text{ for all }i=1,\dots,n-1.
%\]

\

\textbf{Characters for $\g$:}  Recall $[\p(n),\p(n)]$ is a codimension-one ideal of $\p(n)$ with one-dimensional even quotient.  The even irreducible representations of this quotient are indexed by the complex numbers, and pullback to multiples of the representation of highest weight $\omega=\epsilon_1+\dots+\epsilon_n$.  

\begin{prop}\label{p_std}
	The standard module $P_{n|n}$ is never spherical.  The parity shift $\Pi P_{n|n}$ is spherical exactly with respect to (up to conjugacy) Borels $\b$ with $\b_{\ol{0}}=\b_{\ol{0}}^{st}$, and $\epsilon_1+\epsilon_i$ positive for all $i$.
\end{prop}
\begin{proof} This is seen from the matrix presentation in \eqref{p_pres}.\end{proof}

We may now put aside the case when $n=1$:
\begin{prop}
	Up to equivalence, the only non-trivial indecomposable spherical module for $\p(1)$ is $\Pi P_{1|1}$.  The stabilizer of a spherical vector is trivial.
\end{prop}
\begin{proof}The proof is straightforward and thus omitted.\end{proof}

We now assume $n\geq 2$.

\begin{prop}\label{p_sym} The module $S^2P_{n|n}$ is indecomposable.  If $n>2$, it has simple socle $L_{\b^{st}}(-\epsilon_{n-1}-\epsilon_n)$ with a one-dimensional odd quotient.  The form $q$ induces an isomorphism $(S^2P_{n|n})^*\cong\Lambda^2V$, and we have $L(-\epsilon_{n-1}-\epsilon_n)^*\cong L(-2\epsilon_n)$.
	
	In particular, if $n>2$, $\dim^s L(-\epsilon_{n-1}-\epsilon_n)=\dim^s L(-2\epsilon_n)=(n^2|n^2-1)$, and neither module nor its parity shift is spherical.
\end{prop}
\begin{proof}
	For the statement about $S^2P_{n|n}$, see Lemma 2.1.2 of \cite{balagovic2016translation} or the end of \cref{App_sym}.
	
	We have $S^2P_{n|n}^*\cong S^2\Pi P_{n|n}\cong \Lambda^2P_{n|n}$.  We compute directly that the highest weight of $\Lambda^2P_{n|n}$ is $-2\epsilon_n$.
	
	Finally, the statement about dimensions is clear.  Because the maximum of $\dim\b_{\ol{1}}$ over all Borel subalgebras is $n(n+1)/2$, and we have $(n(n+1)/2)<n^2-1$ for $n>2$, the odd part of  $(\Pi)L(-\epsilon_{n-1}-\epsilon_n)$ or $(\Pi)L(-2\epsilon_n)$ will always be too large to be spherical.
\end{proof}

Now we assume $n\geq 3$.

\textbf{Even spherical dominant weights for $\g_{\ol{0}}\cong\g\l(n)$:}
\[
\epsilon_1,\ \ \ 2\epsilon_1,\ \ \ -\epsilon_n,\ \ \ -2\epsilon_n,\ \ \ \epsilon_1+\epsilon_2\ \ \ -\epsilon_{n-1}-\epsilon_n
\]

\textbf{Candidate even weights:} Same as above.

\

\textbf{Check for (numerical) sphericity of $L(\lambda)$:}
\begin{itemize}
	\item[$\bullet$]  The cases of $\lambda=-\epsilon_n,-2\epsilon_n,-\epsilon_{n-1}-\epsilon_n$ were dealt with in \cref{p_std} and \cref{p_sym}.
	
	\item[$\bullet$]  If $\lambda=\epsilon_1$, we compute $L(\lambda)^*\cong L(-n\epsilon_n-\omega)$.  Hence for $n\geq 4$ the even part cannot be spherical, while when $n=3$ the dimension is the same as that of $L(-2\epsilon_n)$, and it cannot be spherical.
	
	\item[$\bullet$]  If $\lambda=2\epsilon_1$, then $L(\lambda)^*\cong L(-(n+1)\epsilon_n-\omega)$, so the even part is never spherical for $n\geq3$.
	
	\item[$\bullet$] If $\lambda=\epsilon_1+\epsilon_2$, then $L(\lambda)^*=L((1-n)\epsilon_n-\omega)$.  Hence for $n\geq 5$ the even part is not spherical, while for $n=4$ the dimension is the same as that of $L(-2\epsilon_1)$, so it cannot be spherical.  If $n=3$, $\epsilon_1+\epsilon_2-\omega=-\epsilon_3$, so this falls under \cref{p_std}.
	
\end{itemize}

\

\textbf{Candidate odd highest weights:}
\[
-\epsilon_1,\ \ \ -2\epsilon_n,\ \ \ -\epsilon_{n-1}-\epsilon_n,\ \ \ -2\epsilon_{n-1}-2\epsilon_n (n=3)
\]
\[
-\epsilon_{n-2}-\epsilon_{n-1}-\epsilon_n\ (n\geq 4)\ \ \ -\epsilon_{n-3}-\epsilon_{n-2}-\epsilon_{n-1}-\epsilon_{n}\ (n\geq 5)
\]

\

\textbf{Check for (numerical) sphericity of $\Pi L(\lambda)$:}
\begin{itemize}
	\item The cases of $\lambda=-\epsilon_n,-2\epsilon,-\epsilon_{n-1}-\epsilon_n$ where dealt with in \cref{p_std} and \cref{p_sym}.
	
	\item If $\lambda=-2\epsilon_{n-1}-2\epsilon_n$ with $n=3$, this is equivalent to $\Pi L(2\epsilon_1)$, which was already dealt with.
	
	\item If $\lambda=-\epsilon_{n-2}-\epsilon_{n-1}-\epsilon_n$ with $n\geq 4$.  This is equivalent to the module $\Pi L(\epsilon_1+\dots+\epsilon_{n-3})$, whose dual is $\Pi L(-4\epsilon_n-\omega)$.  The odd part is then too large to be spherical.
	
	\item If $\lambda=-\epsilon_{n-3}-\epsilon_{n-2}-\epsilon_{n-1}-\epsilon_n$ with $n\geq 5$, this is equivalent to the module $\Pi L(\epsilon_1+\dots+\epsilon_{n-4})$, whose dual is $\Pi L(-5\epsilon_n-\omega)$.  Again the odd part is too large.
\end{itemize}

\subsection{$\p(2)$}

\textbf{Candidate even weights:}
\[
-\epsilon_2, \ \ \ \epsilon_1-\epsilon_2
\]

\textbf{Check for (numerical) sphericity of $L(\lambda)$:}
\begin{itemize} 
	\item The cases $L(-\epsilon_2),\Pi L(-\epsilon_2)$ were covered in \cref{p_std}.
	
	\item The module $L(\epsilon_1-\epsilon_2)$ is the socle of the adjoint representation, represented explicitly as the matrices
	\[
	\begin{bmatrix} A & B\\0 & -A^{t}\end{bmatrix}
	\]
	where $\tr(A)=0$ and $B^t=B$.  An explicit computation shows that neither this nor its parity shift is spherical. Hence neither $L(\epsilon_1-\epsilon_2)$ nor $\Pi L(\epsilon_1-\epsilon_2)$ is spherical.
\end{itemize}

\

\textbf{Candidate odd weights:} 
\[
-\epsilon_2,\ \ \ \epsilon_1-\epsilon_2
\]

\

\textbf{Check for (numerical) sphericity of $\Pi L(\lambda)$:} The case $\lambda=-\epsilon_2$ is covered in \cref{p_std}, and $\lambda=\epsilon_1-\epsilon_2$ was discussed in the even weight check above.

\

\textbf{(Numerically) spherical irreducibles for $\p(n)$, $n\geq 2$:} The only non-trivial spherical irreducible is $\Pi P_{n|n}$.  The stabilizer is $\k=\p(n-1)\ltimes\Pi L(-\epsilon_{n-1})$.  The numerically spherical irreducibles also include $\Pi\C$ and $P_{n|n}$.

\

\textbf{Spherical Indecomposables:}  For $\p(2)$, up to equivalence there is one non-trivial extension of one-dimensional modules, which is equivalent to the $\p(1)$-module $P_{1|1}$.  There are no extensions of $(\Pi)P_{2|2}$ by a one-dimensional module, and any extensions of $(\Pi)P_{2|2}$ by a twist of $(\Pi)P_{2|2}$ have too large an odd part to be spherical.  This deals with the case $n=2$.

There are no non-trivial extensions of 1-dimensional modules, or of two modules equivalent to $(\Pi)P_{n|n}$ for $\p(n)$ when $n\geq 3$.  Therefore we need to determine when there are extensions between $(\Pi)P_{n|n}$ and a one-dimensional module.  

The weights of $(\Pi)P_{n|n}$ are $\pm\epsilon_i$, while the weights of any one-dimensional module are multiples of $\omega$.  Because our odd roots are all of the form $\pm(\epsilon_i+\epsilon_j)$, the only time such an extension could exist is when $n=3$.  Further, the extension would need to appear in either the thin Kac module $\nabla(\omega)$ or the thick Kac module $\Delta(-\omega)$ by our weight restrictions  (see \cite{balagovic2016translation} for more on these modules).

We see that neither $P_{3|3}$ nor its parity shift appear in the thick Kac module
\[
\Delta(-\omega)=\UU\p(3)\otimes_{\UU(\p(3)_{-1}\oplus\p(3)_{0})}\C_{-\omega}
\]

On the other hand, the thin Kac module on $\C_{\omega}$, i.e. 
\[\nabla(\omega)=\UU\p(3)\otimes_{\UU(\p(3)_{0}\oplus\p(3)_{1})}\C_{\omega}
\]
provides us with a module with the following socle filtration:
\renewcommand{\arraystretch}{1.4}
\[
\nabla(\omega)=\begin{tabular}{|c|}
\hline 
$\C_{\omega}$\\
\hline
$\Pi P_{3|3}$\\
\hline
$\Pi\C_{-\omega}$\\
\hline
\end{tabular}.
\]
\renewcommand{\arraystretch}{1}
Since the space of extensions between any two simple modules for $\p(3)$ is always at most one-dimensional, the only extensions of $(\Pi)P_{3|3}$ by a one-dimensional module will appear in a subquotient of the above module or its parity shift.  Of those, the ones which are non-irreducible and spherical are $\nabla(\omega)$, $\operatorname{rad}\nabla(\omega)$, and $\nabla(\omega)/\operatorname{soc}\nabla\omega$.  We note that $\nabla(\omega)$ is in fact a restriction of $P_{4|4}$.  Information about Borels for which sphericity is achieved and stabilizers of spherical vectors is not very revealing, and thus is omitted.

%%The description of Borels for which sphericity is achieved, and the stabilizer of spherical vectors is a bit messy, so in lieu of that information we give an explicit presentation of the representation:
%%\[
%%\begin{bmatrix}a_{11} & a_{12} & a_{13} & b_{11} & b_{12} & b_{13}\\ a_{21} & a_{22} & a_{23} & b_{12} & b_{22} & b_{23}\\ a_{31} & a_{32} & a_{33} & b_{13} & b_{23} & b_{33}\\ 0 & c_{12} & c_{13} & -a_{11} & -a_{21} & -a_{31} \\ -c_{12} & 0 & c_{23} & -a_{12} & -a_{22} & -a_{32} \\-c_{13} & -c_{23} & 0 & -a_{13} & -a_{23} & -a_{33}\end{bmatrix}\mapsto 
%%\]
%%\[
%%\begin{bmatrix}a_{11}+a_{22}+a_{33} & 0 & 0 & 0 & 0 & 0 & 0 & 0\\ 0 & -a_{11} & -a_{21} & -a_{31} & 0 & 0 & c_{12} & c_{13}\\ 0 & -a_{12} & -a_{22} & -a_{32} & 0  & -c_{12} & 0 & -c_{23}\\ 0 & -a_{13} & -a_{23} & -a_{33} & 0 & -c_{13} & -c_{23} & 0\\ 0 & c_{23} & -c_{13} & c_{12}  & -a_{11}-a_{22}-a_{33} & 0 & 0 & 0 \\ c_{23} & b_{11} & b_{12} & b_{13} & 0 & a_{11} & a_{12} & a_{13} \\ -c_{13} & b_{12} & b_{22} & b_{23} & 0 & a_{21} & a_{22} & a_{23} \\ c_{12} & b_{13} & b_{23} & b_{33} & 0 & a_{31} & a_{32} & a_{33}\end{bmatrix}
%%\]
%In particular, one sees that this representation is a restriction of $\Pi P_{4|4}$.

\section{$\q(n)$ Case}\label{q_case}

Let $\g=\q(n)$.  For a more in-depth treatment of this algebra we refer the reader to \cite{cheng2012dualities}.

\textbf{Notation:} We present $\g$ as the subalgebra of $\g\l(n|n)$ consisting of matrices 
\[
\g=\left\{\begin{bmatrix}A & B\\B & A\end{bmatrix}:A,B\in\g\l(n)\right\}
\]
Let
\[
\h=\begin{bmatrix}D & D'\\D' & D\end{bmatrix}
\]
where $D,D'$ are arbitrary diagonal matrices.  

We write $Q_{n|n}$ for the standard module of $\q(n)$.

\textbf{Root system:} Let $\epsilon_1,\dots,\epsilon_n$ be the usual basis of $\h_{\ol{0}}^*$.  Write $\Delta$ for the non-zero weights of the adjoint action of $\h_{\ol{0}}$ on $\g$, i.e. the roots of $\g$.  We have
\[
\Delta=\{\epsilon_i-\epsilon_j\text{ for }1\leq i\neq j\leq n\},
\]
and all roots spaces are $(1|1)$-dimensional.

\textbf{Standard Borel:} Choose for positive system $\epsilon_1-\epsilon_2,\dots,\epsilon_{n-1}-\epsilon_n$.  Write $\b^{st}$ for the corresponding Borel subalgebra.  Note that for $\q(n)$ all Borel subalgebras are conjugate.  Notice also that $\dim\b_{\ol{1}}=n(n+1)/2<n^2$ for $n\geq 2$.   

%\textbf{$\b^{st}$-dominant weights:}  
%\[
%\lambda=s_1\epsilon_1+\dots+s_n\epsilon_n\ \text{ with }s_i\in\C,\ \text{ and such that }s_i-s_{i+1}\in\Z_{\geq0},
%\] 
%and if $s_i=s_{i+1}$ for some $i$ then $s_i=0$.  We say $\lambda$ is typical if $s_i\neq-s_j$ for all $i,j$.

\subsection{Spherical weights for $\g_{\ol{0}}$:}
	\[
\epsilon_1,\ \ \ 2\epsilon_1, \ \ \ -\epsilon_n, \ \ \ -2\epsilon_n
	\]
	\[
	\epsilon_1+\epsilon_2, \ \ \ -\epsilon_{n-1}-\epsilon_{n}
	\]
	
\subsection{$\q(1)$}

\begin{prop} For $\q(1)$, the non-trivial spherical indecomposable modules are $\Ind_{\q(1)_{\ol{0}}}^{\q(1)}\C_{t\epsilon_1}$, where $t\in\C$ is arbitrary.  For $t\neq 0$ these modules are all equivalent to $Q_{1|1}$, and the stabilizer of a spherical vector is trivial.  When $t=0$, $\Ind_{\q(1)_{\ol{0}}}^{\q(1)}\C_{0}$ is equivalent to $U^{1|1}$.
\end{prop}

\begin{proof}Omitted.\end{proof}

\subsection{$\q(n)$, $n\geq 3$}

\textbf{Candidate weights:} (all weights below are length one, hence satisfy $\Pi L(\lambda)\cong L(\lambda)$)
\[
\epsilon_1, \ \ \ 2\epsilon_1,\ \ \ -\epsilon_n,\ \ \ -2\epsilon_n 
\]

\textbf{Check for (numerical) sphericity of $L(\lambda)$:}  We have $L(\epsilon_1)\cong Q_{n|n}$ and $L(-\epsilon_n)\cong Q_{n|n}^*$, and each are spherical by a straightforward check- in fact, they are equivalent..  The module $L(2\epsilon_1)$ is $S^2(Q_{n|n})$, so has odd dimension $n^2$, which is too large.  Similarly $L(-2\epsilon_n)$ is $S^2(Q_{n|n}^*)$, so cannot be spherical.

\textbf{Numerical spherical irreducibles for $\q(n)$, $n\geq 3$:} Up to equivalence, the only non-trivial numerically spherical irreducible is $Q_{n|n}$ and it is spherical.  The stabilizer of a spherical vector is $\q(n-1)\ltimes(Q_{n-1|n-1})^*$.

\textbf{Spherical indecomposables for $\q(n)$, $n\geq 3$:} For $\q(n)$, there are no extensions of $\C$ by $\C$, however the surjective algebra homomorphism $\q(n)\to\C^{0|1}$ gives rise to a non-trivial extension of $\Pi\C$ by $\C$.  This module is spherical and equivalent to $U^{1|1}$.  

Since $\C$, $Q_{n|n}$, and $Q_{n|n}^*$ have distinct central characters, there are no extensions between them.  Any extension of $Q_{n|n}$ by itself cannot be spherical because its even part will not be spherical.  Hence we have found all spherical indecomposables of $\q(n)$ for $n\geq 3$.

\subsection{$\q(2)$}

\textbf{Candidate weights:}
\[
(t+1)\epsilon_1+t\epsilon_2,\ \ \ (t+2)\epsilon_1+t\epsilon_2,\text{  where }t\in\C
\]
We need to consider the irreducible representations with these given highest weights as well as their parity shifts when the length of the weight is two.  Note that the above weights are always typical except for $\frac{1}{2}\epsilon_1-\frac{1}{2}\epsilon_2$ and $\epsilon_1-\epsilon_2$.

\textbf{Check for (numerical) sphericity of $L(\lambda)$ and $\Pi L(\lambda)$:}
\begin{itemize}
	\item If $\lambda=\epsilon_1-\epsilon_2$ then up to parity shift, $L(\epsilon_1-\epsilon_2)\cong [\q(2),\q(2)]/\C I_{2|2}$, which is $(3|3)$ dimensional.  By a direct computation neither this nor its parity shift is spherical.
	
	\item Suppose $\lambda=\frac{1}{2}\epsilon_1-\frac{1}{2}\epsilon_2$. Then $L(\lambda)$ is isomorphic up to a parity shift to the representation coming from the map of algebras $\q(2)\to\p(2)$ which induces the following exact sequence:
	\[
	0\to\C I_{2|2}\to\q(2)\to[\p(2),\p(2)]\to0
	\]
	Therefore we can understand this representation, up to a parity shift, as the restriction of the action of $\p(2)$ on $P_{2|2}$ to its derived subalgebra.
	
	\item If $\lambda=(t+2)\epsilon_1+t\epsilon_2$ for $t\neq-1$, then the character formula for $L(\lambda)$ tells us that it will be $(4|4)$ dimensional, so neither it nor its parity shift can be spherical, having too large an odd part.
	
	\item Finally, suppose $\lambda=(t+1)\epsilon_1+t\epsilon_2$, with $t\neq -1/2$.  Then by direct computation both $L(\lambda)$ and $\Pi L(\lambda)$ are spherical (note that $L(\lambda)\cong\Pi L(\lambda)$ if and only if $t=0$ or $t=-1$).
\end{itemize}

\textbf{Spherical irreducibles for $\q(2)$:}  We compute $S^d(Rep^*)$ with respect to $(\b^{st})^{op}$.  This leads to a canonical identification $L_{\b^{st}}(\lambda)^*\cong L_{(\b^{st})^{op}}(-\lambda)$ when the length of $\lambda$ is $2$.

\renewcommand{\arraystretch}{1.5}
\[\begin{tabular}{ |c|c|c|c| } 
\hline 
Rep & $\dim^s$ & Stabilizer \\
\hline
$L((t+1)\epsilon_1+t\epsilon_2)$, $t\neq -1/2$& $(2|2)$ & $\q(1)\ltimes \Ind_{\q(1)_{\ol{0}}}^{\q(1)}\C_{2t+1}$ \\
\hline 
$\Pi L((t+1)\epsilon_1+t\epsilon_2)$, $t\neq -1/2$ & $(2|2)$ & $\q(1)\ltimes \Ind_{\q(1)_{\ol{0}}}^{\q(1)}\C_{2t+1}$ \\
\hline 
$\text{Res}_{[\p(2),\p(2)]}P_{2|2}$ & $(2|2)$ & $\q(1)\ltimes\Pi(\Ind_{\q(1)_{\ol{0}}}^{\q(1)}\C_0)$\\
\hline
\end{tabular}
\]
\renewcommand{\arraystretch}{1}

\

\textbf{Spherical indecomposables for $\q(2)$:}  Again, the quotient map $\q(2)\to\C^{0|1}$ gives rise to a module equivalent to $U^{1|1}$.  The modules $(\Pi)L((t+1)\epsilon_1+t\epsilon_2)$ for $\lambda\neq 0,-1/2,,-1$ are projective hence have no extensions with other modules.  For $\lambda=-1/2$, it has the same central character as the trivial module but its weights prevent any extensions between them.  Finally, $Q_{2|2}$ and $Q_{2|2}^*$ and the trivial module again have distinct central characters, so we get new spherical indecomposable representations.

\section{Computations of Symmetric Powers}\label{App_sym}

We begin by stating
\begin{prop}
	Let $U^{1|1}$ be the representation of the one-dimensional odd abelian algebra as described in the introduction.  Then $S^d(U^{1|1})^*\cong \Pi U^{1|1}$ for all $d$.
\end{prop}

\begin{proof}
	Omitted.
\end{proof}

\subsection{$\g\l(m|n)$:}  The computations of $S^d(V^*)$ for a spherical irreducible of $\g\l(m|n)$ from \cref{gl_std_reps} use Schur-Weyl duality as in \cite{howe1995perspectives}.  Most computations were also discussed in \cite{sahi2018capelli}.  We do not rewrite them here.

\subsection{$\g\l(1|2)$:}

First we compute $S^d(K(t\epsilon_1)^*)$ for $d\geq 1$ as a $\g\l(1|2)$-module.  We have that $K(t\epsilon_1)^*\cong K((2-t)\epsilon_1-\delta_1-\delta_2)$.  Then with respect to $\b^{\delta\delta\epsilon}$, the highest weights functions in $S^d(K(t\epsilon_1)^*)$ have weight 
\[
\lambda_{i,j}=(2i-(i+j)t)\epsilon_1-i(\delta_1+\delta_2)
\]
where $i+j=d$, $i\geq 0$, and $j>0$.  The weight $\lambda_{i,j}$ is then atypical if and only if 
\[
t=\frac{i}{i+j}=\frac{i}{d}\ \ \text{ or }\ \ t=\frac{i+1}{i+j}=\frac{i+1}{d}
\]
This can only happen if $t\in\Q\cap[0,1]$.  Therefore, if $t\notin\Q\cap[0,1]$, we have
\begin{align}\label{t_generic}
S^d(K(t\epsilon_1)^*)=\bigoplus\limits_{\makecell{i+j=d\\ i\geq 0, j>0}}L_{\b^{\delta\delta\epsilon}}(\lambda_{i,j})
\end{align}
Now suppose that $t=\frac{m}{n}$ with $m,n\in\Z$, $0\leq m\leq n$, and $(m,n)=1$.  Then if $n$ does not divide $d$, we get the decomposition as in \ref{t_generic}.  

If $n$ divides $d$, we may write $d=nk$, where $k>0$.  If $i+j=d$, then $\lambda_{i,j}$ is atypical if and only if $i=mk$ or $i=mk-1$.  If $t\neq 0,1$ then $S^d((K(\frac{m}{n}\epsilon_1)^*)$ is projective, since it is a summand of the projective module $(K(\frac{m}{n}\epsilon_1)^*)^{\otimes d}$.  Therefore in this case we find that
\[
S^{nk}(K(\frac{m}{n}\epsilon_1)^*)=P(mk\Ber)\oplus\bigoplus\limits_{\makecell{i+j=nk\\ 0\leq i<nk, i\neq mk,mk+1}}K(\lambda_{i,j})
\]
where $P(mk\Ber)$ is the projective cover of the one-dimensional even module of weight $mk\Ber$.

If $t=0$ then $m=0$ and $n=1$.  The weight $\lambda_{i,j}=2i\epsilon_1-i(\delta_1+\delta_2)$ is atypical only when $i=0$, $j=k$.  Therefore we get
\[
S^{k}(K(0)^*)=K_{1|2}(0)^*\oplus\bigoplus\limits_{\makecell{i+j=k\\ 0<i<k}}L_{\b^{\delta\delta\epsilon}}(\lambda_{i,j})
\]
The first summand is obtained by using that that $K(0)^*$ has an even $\g$-invariant vector, so multiplication by powers of it define injective homomorphisms $K(0)^*\to S^{k}(K(0)^*)$ for all $k$.

Finally, if $t=1$ then $m=n=1$.  The weight $\lambda_{i,j}=(i-j)\epsilon_1-i(\delta_1+\delta_2)$ is atypical if and only if $j=1, i=k-1$.  Therefore we get
\[
S^{k}(K(\epsilon_1)^*)=(K_{1|2}(\epsilon_1)^*)_{(k-1)Ber}\oplus\bigoplus\limits_{\makecell{i+j=k\\ 0\leq i<k-1}}L_{\b^{\delta\delta\epsilon}}(\lambda_{i,j})
\]
The first summand can be obtained by observing we have an even $\g$ semi-invariant derivation on functions coming from the even semi-invariant element of $K(\epsilon_1)$.  Powers of it define surjective $\g$ semi-equivariant homomorphisms $S^k(K(\epsilon_1)^*)\to K(\epsilon_1)^*$ for all $k$.

\

\subsection{$S^d(\Pi K(t\epsilon_1)^*)$, $t\neq 0$ and $d>0$}  Here, with respect to the $\delta\epsilon\delta$-Borel, the highest weight is $(1-t)\epsilon_1-\delta_2$, and the $d$th power of the highest weight vector will be a highest weight vector of weight
\[
\mu_d=d(1-t)\epsilon_1-d\delta_2
\]
When $t\neq 1$, these weight are typical.  It follows that
\[
S^d(\Pi K(t\epsilon_1)^*)\cong L_{\b^{\delta\epsilon\delta}}(\mu_d)
\]
When $t=1$, $\Pi K(t\epsilon_1)$ has socle $\Pi\C_{-\Ber}$.  Therefore, there is an odd $\g$ semi-invariant derivation on functions, which defines non-zero $\g$ semi-equivariant homomorphisms $S^d(\Pi K(t\epsilon_1)^*)\to S^{d-1}(\Pi K(t\epsilon_1)^*)$ such that the composition of two is zero.  Since $S^d(\Pi K(t\epsilon_1)^*)$ also contains a unique highest weight vector, of weight $\mu_d=-d\delta_2$, the socle must be $L_{\b^{\delta\epsilon\delta}}(-d\delta_2)$, and the derivation must vanish on the socle.  By the structure of projectives for an atypical block of $\g\l(1|2)$-modules, it follows we must have
\[
S^d(\Pi K(t\epsilon_1)^*)\cong K_{1|2}(\epsilon_1-\delta_1-d\delta_2)
\]

\subsection{$\o\s\p(m|2n)$}  

Recall that $\Lambda^d(OSP_{m|2n})$ is irreducible for all $d$ as a $\o\s\p(m|2n)$-module.  Since
\[
S^d(\Pi OSP_{m|2n})\cong \Pi^d\Lambda^d(OSP_{m|2n}),
\]
all symmetric powers of $\Pi OSP_{m|2n}$ are irreducible.  

\

The computation of $S^dV^*$ for the $\o\s\p(2|4)$-module $V=L(\delta_1+\delta_2-\epsilon_1)$ is given in \cite{sahi2018capelli}.
%Now consider the $\o\s\p(2|4)$-module $V=L(\delta_1+\delta_2-\epsilon_1)$.  We express $S^d(V^*)$ as a $\o\s\p(2|4)\times\C$-module with respect to the Borel $\b^{(-\epsilon)\delta\delta}$, following the presentation given in \cite{sahi2018capelli}.  There it was shown that $S^\bullet(V^*)$ is completely reducible with the lattice of highest weights given by $\Lambda^+(V):=\langle \epsilon_1+\delta_1+\delta_2+\zeta,2\epsilon_1+2\zeta,(3s-4)\epsilon_1+s\zeta\rangle_{s\in\Z_{\geq 2}}$, where $\zeta$ is the weight of the central action.  For $\lambda\in\Lambda^+(V)$, write $\deg(\lambda)$ for the non-negative integer $t\in\Z_{\geq0}$, where the coefficient of $\zeta$ in $\lambda$ is $t$.  Write $\Lambda^+(V)_d$ for the set of $\lambda\in\Lambda^+(V)$ with $\deg(\lambda)=d$.  Then we have
%\[
%S^d(V^*)=\bigoplus\limits_{\lambda\in\Lambda^+(V)_d}L(\lambda)
%\]

\subsection{Computation of $S^\bullet(OSP_{m|2n})$}\label{super_sphere_comp}

Write $V=OSP_{m|2n}$ for the standard representation of $\g=\o\s\p(m|2n)$, and assume $m\geq2$ so that $V$ is spherical for $\g\o\s\p(m|2n)$.  For each $d\geq 0$ we find the structure of $S^d V$ as a $\g$-module.  This is related to the notion of skew-symmetric harmonic polynomials, and the proof uses ideas from the classical theory of harmonic polynomials as functions on the sphere.  For more on the classical story of harmonic polynomials, see Chapter III of \cite{helgason1984groups}.  This question was studied in the real case in \cite{zhang2008orthosymplectic} and further in \cite{coulembier2012orthosymplectic} in the context of supersymmetric harmonic analysis.  Many of the results we present are present in those two papers, although the methods used are different.  

\subsection{Setup} Let $(-,-)\in S^2V^*$ be a non-degenerate $\g$-invariant supersymmetric form on $V$.  Then we have an induced isomorphism of $\g$-modules $V\cong V^*$ and a corresponding dual element to the form, $\omega\in S^2V$.  The form $(-,-)$ gives rise to a non-degenerate, supersymmetric $\g$-invariant form on each symmetric power $S^dV$, which we also denote by $(-,-)$.  Let $\Omega\in\End(S^\bullet V)$ be the adjoint to left multiplication by $\omega$, i.e.
\[
(\Omega x,y)=(x,\omega y)
\]
for all $x,y\in S^\bullet V$. This is the Laplacian operator on functions.  Since multiplication by $\omega$ is injective, $\Omega$ is surjective.  Let $H=[\Omega,-\omega]$.  Then for each $d$ we have $\g$-module endomorphisms
\[
H:S^dV\to S^dV, \ \ \ \Omega:S^{d}V\to S^{d-2}V, \ \ \ L_\omega:S^{d}V\to S^{d+2}V
\]
where we denote left multiplication by $\omega$ as $L_\omega$.  Further, these three endomorphisms form an $\s\l_2$-triple:
\[
[H,\Omega]=2\Omega,\ \ \ [H,-\omega]=2\omega.
\]
Therefore we have an action of $\s\l_2\times\g$ on $S^\bullet V$.  The operator $H$ takes the specific form
\[
H=(n-r)-\EE
\]
where $r:=m/2$ and $\EE$ is the Euler vector field, i.e. the operator which acts as scalar multiplication by $d$ on $S^dV$.

\subsection{$\s\l_2\times\g_{\ol{0}}$ structure:} By the theory of harmonic polynomials on $\C^m$, we have as an $\s\l_2\times\o(m)$-module
\[
S^\bullet V_{\ol{0}}=\bigoplus\limits_{\ell\geq0}M(-r-\ell)\boxtimes H_{\ell}^{ev},
\]
where $H_{\ell}^{ev}$ is the irreducible $\s\o(m)$-module of harmonic polynomials of degree $\ell$, and we write $M(s)$ for the $\s\l_2$ Verma module of highest weight $s$.  Here, we have
\[
M(-r-\ell)_{-r-\ell}\otimes H_{\ell}^{ev}\sub S^\ell V_{\ol{0}}.
\]
By the theory of skew-symmetric harmonic polynomials on $\C^{2n}$, we have as an $\s\l_2\times\s\p(2n)$-module
\[
\Lambda^\bullet V_{\ol{1}}=\bigoplus\limits_{0\leq j\leq n}L(n-j)\boxtimes W_j
\]
where $W_j$ is the $j$th fundamental representation of $\s\p(2n)$ for $j\geq 1$ and $W_0$ is the trivial representation, and we write $L(s)$ for the irreducible $\s\l_2$-module of highest weight $s$.  Here, we have
\[
L(n-j)_{n-j}\otimes W_j\sub \Lambda^jV_{\ol{1}}
\]
Hence we have, as an $\s\l_2\times\s\o(m)\times\s\p(2n)$-module,
\[
S^\bullet V=S^\bullet V_{\ol{0}}\otimes \Lambda^\bullet V_{\ol{1}}=\bigoplus\limits_{\makecell{\ell\geq 0\\0\leq j\leq n}}(M(-r-\ell)\otimes L(n-j))\boxtimes H_\ell^{ev}\boxtimes W_j.
\]

\begin{lemma}\label{g_0_struc}
	As a $\g_{\ol{0}}$-module, $S^dV/L_{\omega}S^{d-2}V$ is multiplicity-free, self-dual, and every irreducible summand is isomorphic to a module of the form $H_i^{ev}\boxtimes W_j$, where $0\leq j\leq n$.
\end{lemma}

\begin{proof}
	The irreducible factors of $S^dV$ are all isomorphic to a module of the form $H_i^{ev}\boxtimes W_j$ for some $0\leq j\leq n$, and these are all self-dual $\g_{\ol{0}}$-modules.  So it remains to prove $S^dV/L_{\omega}S^{d-2}V$ is multiplicity-free.
	
	By our decomposition as an $\s\l_2\times \g_{\ol{0}}$-module, we can write
	\[
	S^dV=\bigoplus\limits_{i+j=d}\left(\bigoplus\limits_{0\leq \ell\leq\lfloor \frac{i}{2}\rfloor}r_0^\ell H_{i-2\ell}^{ev}\right)\otimes\left(\bigoplus\limits_{\max(0,j-n)\leq k\leq\lfloor\frac{j}{2}\rfloor}\omega_0^{k}W_{j-2k}\right)
	\]
	where $\omega=r_0+\omega_0$, $r_0\in S^2V_{\ol{0}}$, $\omega_0\in\Lambda^2V_{\ol{1}}$.  The extra lower bound condition on $k$ comes from the finite-dimensional structure of the corresponding $\s\l_2$-module.  It follows that each summand can be written uniquely as $r_0^\ell H_{s}^{ev}\otimes \omega_0^{k}W_t$, with $2\ell+s+2k+t=d$.
	
	Suppose another summand isomorphic to this one shows up, e.g. $r_0^{p}H_{s}^{ev}\otimes\omega_0^{q}W_t$, with $2p+s+2q+t=d$.  Then we must have $q\neq k$.  Without loss of generality suppose $q<k$.  Then $r_0^{\ell}H_s^{ev}\otimes \omega_0^{k-1}W_t$ will be a $\g_{\ol{0}}$ summand of $S^{d-2}V$.  Multiplying it by $\omega=r_0+\omega_0$, we learn that modulo $L_{\omega}S^{d-2}V$, $r_0^{\ell}H_s^{ev}\otimes\omega_0^{k}W_t$ is identified with $r_0^{\ell+1}H_s^{ev}\otimes\omega_0^{k-1}W_t$.  
	
	By induction on $|k-q|$, we can identify $r_0^\ell H_{s}^{ev}\otimes \omega_0^{k}W_t$ with $r_0^{p}H_{s}^{ev}\otimes\omega_0^{q}W_t$ modulo $L_{\omega}S^{d-2}V$. This proves the quotient is multiplicity-free.
\end{proof}

%The first observation we deduce from the above decomposition is the following isomorphism of $\g_{\ol{0}}$-modules:
%\[
%S^dV/\omega S^{d-2}V\cong\bigoplus\limits_{\substack{0\leq s\leq 2n\\ \max(\frac{s-n}{2},0)\leq q\leq\lfloor\frac{s}{2}\rfloor}}H_{d-s}^{ev}\boxtimes W_{s-2q}
%\]
\begin{cor}
	As a $\g$-module, $S^dV/L_{\omega}S^{d-2}V$ is multiplicity-free and each composition factor is self-dual.
\end{cor}

\

\subsection{Frobenius Reciprocity} If we consider the action of the supergroup $G=OSP(m|2n)$ on $V$ as a supervariety, the stabilizer of an even vector of length $1$ will be $K=OSP(m-1|2n)$, which gives rise to a closed embedding of $G/K$ into $V$.  In fact we get an identification
\[
G/K\cong \Spec\ S^\bullet V/(1-\omega)
\]
Frobenius reciprocity tells us, in this case, that for an integrable $\g$-module $W$ we have
\begin{align}\label{supersphere}
\Hom_{G}(W,\C[G/K])\cong\Hom_{K}(W,\C)=(W^*)^K
\end{align}
This will be heavily used in what follows.  In particular, we observe that for any $d\in\Z_{\geq0}$, we have a natural injective map:
\[
S^dV\hookrightarrow S^\bullet V/(1-\omega)=\C[G/K],
\]
and because $L_{\omega}$ is injective, we have an isomorphism of $\g$-modules
\begin{align}\label{C[G/K]}
\C[G/K]\cong\colim\limits_{d\to\infty}S^{2d}V\oplus\colim\limits_{d\to\infty}S^{2d+1}V
\end{align}
\

\subsection{$\s\l_2$-module structure:}Let $I=\{n-j,n-j-2,\dots,j-n\}$.  By Prop. 3.12 of \cite{enright1979fundamental}, we have:

\begin{align}\label{enright}
M(-r-\ell)\otimes L(n-j)& = &M\oplus \bigoplus\limits_{\substack{t\in I\\-(r+\ell)+t\in\Z_{\geq 0}\\(r+\ell-t)-2+r+\ell\in I}}P(-(r+\ell)+t)
\end{align}
where for $k\geq 0$ we denote by $P(k)$ the big projective in the block of category $\OO$ for $\s\l_2$ containing $L(k)$, and $M$ is a direct sum of Verma modules.  The structure of $P(k)$ is such that the highest weight is $k$, and the endomorphism 
\[
\Omega L_{\omega}:P(k)_j\to P(k)_j
\]
is an isomorphism if $j\neq -k$, and is the zero map when $j=-k$.   

\begin{cor}\label{sscase}
	If $r$ is a half integer or $r>n$, then as an $\s\l_2$-module $S^\bullet V$ is a direct sum of irreducible Verma modules.  In particular, it is semisimple.
\end{cor}

\begin{proof}
	Our conditions imply that $-(r+\ell)+t\notin\Z_{\geq 0}$ for any integer $t\leq n$.  By \ref{enright}, this implies $M(-r-\ell)\otimes L(n-j)$ is a direct sum of Verma modules of either negative or half-integer highest weight. 
\end{proof}

\textbf{Notation:} Write $H_d:=\ker(\Omega:S^dV\to S^{d-2}V)$ for the space of `harmonic superpolynomials'.  Note that since $\Omega$ is never injective, $H_d\neq 0$ for all $d\geq 0$.

\begin{cor}\label{non-ss-case}
	If $n-r\in\Z_{\geq 0}$, we have 
	\[
	S^\bullet V=M\oplus \bigoplus\limits_{d\leq n-r} P(n-r-d)\otimes H_d
	\]
	where $M$ is a direct sum of Verma modules of negative highest weight.
\end{cor}

\begin{proof}
	By \ref{enright}, it suffices to prove that if $-(r+\ell)+t\geq 0$ for $t\in I=\{n-j,n-j-2,\dots,j-n\}$, then 
	\[
	(r+\ell-t)-2+r+\ell\in I 
	\]
	or, equivalently
	\[
	-(n-j)\leq 2(r+\ell)-2-t\leq n-j.
	\]
	These inequalities follow from the following two inequalities:
	\[
	-(n-j)\leq (r+\ell)-t-1< 0, \ \ \ \ 0\leq r+\ell-1\leq n-j
	\]
	where we are using that $r\geq 1$.
\end{proof}

We now break down our analysis of $S^\bullet V$ into two cases: that when it is a semisimple $\s\l_2$-module, i.e. $n-r\notin\Z_{\geq 0}$, and that when it is not, i.e. $n-r\in\Z_{\geq 0}$.

\subsection{Semisimple Case}  We now suppose that either $r$ is a half integer or $r>n$.  Then by \cref{sscase}, we get that $\Omega L_{\omega}:S^dV\to S^dV$ is an isomorphism for all $d$, and therefore we have
\[
S^dV=H_d\oplus L_{\omega} H_{d-2}\oplus L_{\omega}^2 H_{d-4}\oplus\cdots.
\]

\textbf{Claim}: $H_d$ is irreducible for all $d\geq 0$.

To see this, first observe that
\[
S^dV=H_d\oplus L_{\omega} S^{d-2}V
\]
and therefore by \cref{g_0_struc}, all composition factors of $H_d$ have multiplicity one and are self-dual.  But we also observe by $\s\l_2$-semisimplicity that the form on $S^dV$ is non-degenerate when restricted to $L_{\omega} S^{d-2}V$, and therefore the form must also be non-degenerate when restricted to the complement $H_d$.  Therefore $H_d$ itself is self-dual as a $\g$-module. Because it is also multiplicity-free, by a standard argument this implies $H_d$ is completely reducible.

To show that $H_d$ is actually irreducible, observe that we have shown
\[
\C[G/K]=S^\bullet V/(1-\omega)\cong\bigoplus\limits_{d\geq 0}H_d,
\]
and that $H_d$ is completely reducible.  By \eqref{supersphere} $H_d$ is irreducible if and only if 
\[
(H_d^*)^K\cong H_d^K=1
\]
As a $K$-module, we have $V=V'\oplus\C$, where $V'$ is the standard $K$-module, and $\C$ is the one-dimensional even trivial module.  Therefore, we get the $K$-module decomposition
\[
S^d(V)=S^dV'\oplus S^{d-1}V'\oplus\cdots 
\] 
By Cor 5.3 of \cite{lehrer2017invariants}, the dimension of the space of $K$-invariants in $S^aV'$ is $1$ if $a$ is even, $0$ if $a$ is odd (where we use here that necessarily $m>2$), and hence $\dim S^d(V)^K=\lfloor\frac{d}{2}\rfloor+1$.  On the other hand,
\[
S^d(V)\cong H_d\oplus H_{d-2}\oplus\cdots
\]
Since we must have $\dim H_j^K\geq 1$ for each $j\geq 0$, we obtain that $\dim H_d=1$.  Hence $H_d$ is irreducible.

\subsection{Non-semisimple case: $n-r\in\Z_{\geq 0}$}

\begin{lemma}The map
	\[
	\Omega L_{\omega}:S^dV\to S^dV
	\]
	is an isomorphism if and only if $d<n-r$ or $d>2(n-r)$.
	
	If $n-r+2\leq d\leq 2(n-r)+2$, write $s=d-(n-r)-1$.  Then
	\[
	\ker(\Omega L_{\omega}:S^{d-2}V\to S^{d-2}V)=L_{\omega}^{s-1}H_{d-2s} 
	\]
	In particular $L_{\omega}^{s}H_{d-2s}\sub H_{d}$.
\end{lemma}
\begin{proof}
	This follows from \cref{non-ss-case} and the structure of the $\s\l_2$-modules $P(k)$ for $k\geq 0$.
\end{proof}

It follows that we have, for $0\leq d\leq n-r+1$,
\[
S^dV=H_d\oplus L_\omega H_{d-2}\oplus\cdots .
\]
Further, by following the same proof as in the semisimple case we can again show that each such $H_d$ for $0\leq d\leq n-r+1$ is a simple $\g$-module.  

Now suppose $n-r+2\leq d\leq 2(n-r)+2$.  Using the $\s\l_2$ structure, we may write 
\[
S^{d-2}V=L_{\omega}^{s-1}H_{d-2s}\oplus W_{d-2}
\]
for some complementary $\g$-submodule $W_{d-2}$ with $\Omega L_{\omega}:W_{d-2}\to W_{d-2}$ an isomorphism.  Hence we may write
\[
S^dV=A_{d}\oplus L_{\omega}W_{d-2}
\]
where $A_{d}$ is a $\g$-module complement to $L_{\omega}W_{d-2}$.  In particular, $H_d\sub A_d$.

\

For $d>2(n-r)+2$, $\Omega L_\omega:S^{d-2}V\to S^{d-2}V$ is an isomorphism, so $H_d$ splits off from $S^dV$, and we get a decomposition
\[
S^dV=\bigoplus\limits_{i=0}^{t}L_{\omega}^{i}H_{d-2i}\oplus \bigoplus\limits_{j=t+1}^{\lfloor\frac{d-(n-r)-1}{2}\rfloor} L_{\omega}^jA_{d-2j}
\]
where 
\[
t=\left\lfloor\frac{d-2(n-r)-3}{2}\right\rfloor, \ \ \ \ A_{n-r+1}:=H_{n-r+1}
\] 
Again using Cor. 5.3 of \cite{lehrer2017invariants} and arguments as before, one can show that $H_d$ is irreducible for $d>2(n-r)+2$ when $r\geq 2$, while $H_d$ is the sum of two irreducibles with highest weights $d\epsilon_1$ and $(2n-d)\epsilon_1$ with respect to $\b^{st}$ for $r=1$.  It therefore remains to understand the structure of $A_d$.

\subsection{Structure of $A_d$} We now assume $(n-r)+2\leq d\leq 2(n-r)+2$.  Recall our decomposition
\[
S^dV=A_d\oplus L_{\omega} W_{d-2}
\]
Since the form will be non-degenerate when restricted to $L_{\omega}W_{d-2}$, it will also be non-degenerate on $A_{d}$, and therefore $A_d$ is self-dual.  Further, we observe that $\im L_{\omega}\cap A_{d}=L_{\omega}^{s}H_{d-2s}$, and in fact we have
\[
S^dV/L_{\omega} S^{d-2}V\cong A_{d}/L_{\omega}^{s}H_{d-2s},
\]
so $A_{d}/L_{\omega}^{s}H_{d-2s}$ is multiplicity-free.

By construction, $L_{\omega}^{s}H_{d-2s}\sub H_d\sub A_{d}$.  We get short exact sequences
\[
0\to H_d\to A_d\xto{\Omega}L_{\omega}^{s-1}H_{d-2s}\to0
\]
\[
0\to L_{\omega}^{s}H_{d-2s}\to H_d\to Q_d\to 0
\]
where we have defined $Q_d$ as the quotient $H_d/L_{\omega}^{s}H_{d-2s}$.

Using self-duality of $A_{d}$ we find that $Q_d$ is self-dual, and since $Q_d$ is a submodule of $A_{d}/L_{\omega}^{s}H_{d-2s}$ we get that it is multiplicity-free and each composition factor is self-dual.  Therefore it must be completely reducible.  

Again, by Cor 5.3 of \cite{lehrer2017invariants}, we learn that the $(A_d)^K$ is two-dimensional (even for the case $r=1$), and since $A_d$ is self-dual we have by \ref{supersphere} 
\[
\dim\Hom(A_d,\C[G/K])=2.
\]
Two such linearly independent maps are  
\[
\phi:A_d\sub S^dV\hookrightarrow S^{\bullet} V/(1-\omega) \ \text{ and } \ \psi:A_d\xto{\Omega}L_{\omega}^{s-1}H_{d-2s}\sub S^{d-2}V\hookrightarrow S^\bullet V/(1-\omega)
\]
\

\textbf{Claim:} $A_{d}$ is indecomposable, with irreducible head and socle isomorphic to $H_{d-2s}$.

\

To prove the claim, first notice that the map $A_d\xto{\Omega}L_{\omega}^{s-1}H_{d-2s}$ cannot split, for otherwise $H_{d-2s}^{\oplus 2}$ would be a submodule of $S^dV$, which would contradict \cref{cmult_free}.

Now suppose that $A_d$ split, i.e. we have $A_d=M\oplus N$ for two non-trivial submodules $M$ and $N$, and write $p_M,p_N$ for the projections onto $M$ and $N$ respectively.  Then $\phi\circ p_M$, $\phi\circ p_N$ and $\psi$ would be three linearly independent maps $A_d\to\C[G/K]$, a contradiction.  Therefore $A_d$ is indecomposable.

The fact that $\text{soc}(A_d)\cong H_{d-2s}$ follows from the fact that $Q_d$ is multiplicity-free and each summand is self-dual.  Since $A_d$ is self-dual, the head must also be isomorphic to $H_{d-2s}$.  We now have the following picture of $A_d$, with its socle filtration illustrated:

%By duality, this also implies that $L_{\omega}^{s}H_{d-2s}\hookrightarrow A_{d}$ does not split.
%
%Now suppose that the following short exact sequence splits:
%\[
%0\to Q_d\to A_{d}/L_{\omega}^{s}H_{d-2s}\xto{\Omega}L_{\omega}^{s-1}H_{d-2s}\to 0.
%\]
%Then because $A_d$ doesn't split, we would obtain a necessarily non-trivial extension of $H_{d-2s}$ by $H_{d-2s}$.  However no such extensions exist.  Therefore this sequence cannot split.  It follows by duality that 
%\[
%0\to L_{\omega}^{s}H_{d-2s}\to H_d\to Q_d\to 0
%\]
%does not split.  
%
%We have shown that $A_d$ must take the form $A_d=A_d'\oplus U$, where $U\sub Q_d$ is a semi-simple, completely reducible module with no factor isomorphic to $H_{d-2s}$, and $A_d'$ is indecomposable with a three-step socle filtration, with socle and head isomorphic to $H_{d-2s}$.  However, if $U\neq 0$, then projection onto $U$ will give a new non-trivial map $A_d\to\C[G/K]$, forcing $\Hom(A_d,\C[G/K])$ to be too large.  It follows we must have $A_d=A_d'$.  Therefore, we have the following picture of the module $A_d$, with its socle filtration illustrated:
\renewcommand{\arraystretch}{1.4}
\[
A_d=\begin{tabular}{ |c| }
\hline
$H_{d-2s}$\\
\hline
$Q_d$\\
\hline
$H_{d-2s}$\\
\hline 
\end{tabular}
\]
\renewcommand{\arraystretch}{1}
It remains to understand $Q_d$.  We split our analysis into the cases of $r\geq 2$ and $r=1$.  

\textbf{Case when $r\geq 2$.}  By our description of
\[
S^dV/L_{\omega}S^{d-2}V\cong A_d/L_\omega^{s}H_{d-2s}
\]
as a $\g_{\ol{0}}$-module in \cref{g_0_struc}, we see that the $\g_{\ol{0}}$-dominant weights with respect to the standard even Borel of $\g_{\ol{0}}$ are all of the form $t\epsilon_1+\delta_1+\dots+\delta_i$ for some $t\geq 0$ and some $i$.  

If we choose the Borel corresponding to the $\epsilon\delta$-sequence $\delta\cdots\delta\epsilon\cdots\epsilon$, we see that the only such weights which are $\g$-dominant with respect to this Borel are ones of the form $\delta_1+\dots+\delta_n+t\epsilon_1$ for some $t\geq 0$.  When we change via odd reflections to the Borel with $\epsilon\delta$ sequence $\epsilon\cdots\epsilon\delta\cdots\delta$, this highest weight becomes $(t+n)\epsilon_1$.  Call this latter Borel $\b$.  

We learn therefore, that with respect to the Borel $\b$, the irreducibles which can appear in $Q_d$ must all have highest weight $t\epsilon_1$ for some $t\geq 0$.

Since $A_d$ is indecomposable, any irreducible factors which show up in it must have the same central character.  The module $H_{d-2s}$ has highest weight $(d-2s)\epsilon_1$ with respect to $\b$, and the only other highest weight of the form $t\epsilon_1$ with the same central character is $d\epsilon_1$.  It follows that either $Q_d=L_{\b}(d\epsilon_1)$ or $Q_d=0$.  But again, if $Q_d=0$ then $A_d$ would give a non-trivial extension of $H_{d-2s}$ by itself which does not exist. This gives the structure of $A_d$ when $r\geq 2$:

\renewcommand{\arraystretch}{1.4}
\[
A_d=\begin{tabular}{ |c| }
\hline
$L_{\b}((d-2s)\epsilon_1)$\\
\hline
$L_{\b}(d\epsilon_1)$\\
\hline
$L_{\b}((d-2s)\epsilon_1)$\\
\hline 
\end{tabular}
\]
\renewcommand{\arraystretch}{1}

\textbf{Case when $r=1$}.  Notice that our decomposition of $S^dV$ for $d$ large implies, by \ref{C[G/K]},
\[
\C[G/K]=\bigoplus\limits_{i=n-r+1}^{2(n-r)+2}A_d\oplus \bigoplus\limits_{d>2(n-r)+2}H_d
\]
so $A_d$ is a direct summand of $\C[G/K]$. On the other hand, $\C[G/K]=\Ind_{K}^{G}\C$, where here $K=OSP_{1|2n}$.  Since the category of finite-dimensional representations of $K$ is semi-simple, by Frobenius reciprocity we obtain that $\C[G/K]$ must be a direct sum of injective $\g$-modules.  

In particular, $A_d$ must be itself be a sum of injective modules.  Because we have shown it is indecomposable with socle $L_{\b^{st}}((d-2s)\epsilon_1)$ with respect to the standard Borel $\b^{st}$ of $\o\s\p(2|2n)$, it must be the injective hull of this irreducible module.  The socle filtration of this module is:

\renewcommand{\arraystretch}{1.4}
\[
A_d=\begin{tabular}{ |c| }
\hline
$L_{\b^{st}}((d-2s)\epsilon_1)$\\
\hline
$L_{\b^{st}}(d\epsilon_1)\oplus L_{\b^{st}}(-\epsilon_1+\delta_1+\dots+\delta_{d-2s+1})$\\
\hline
$L_{\b^{st}}((d-2s)\epsilon_1)$\\
\hline 
\end{tabular}
\]
\renewcommand{\arraystretch}{1}

\subsection{Computation of $S^\bullet P_{n|n}$}\label{peri_sym_power}

In this section we compute the algebra of functions on $\Pi P_{n|n}$ as a $\p(n)$-module.

Let $V=P_{n|n}$.  We choose for Borel $\b=\begin{bmatrix} A& B\\0 & -A^{t}\end{bmatrix}$, where $A$ is upper triangular and $B$ is symmetric.  Recall that $\Pi V$ is spherical with respect to $\b$.  Therefore, we will compute highest weights of $S^\bullet(\Pi V^*)$ with respect to $\b$.

We have $\Pi V^*\cong V\cong L_{\b}(\epsilon_1)$, so $S^d(\Pi V^*)\cong S^d(V)$.  Our non-degenerate odd $\p(n)$-invariant form, which we view as an odd $\p(n)$-module homomorphism
\[
q:S^2V\to\C,
\]
gives odd $\p(n)$-module homomorphisms 
\[
q:S^dV\to S^{d-2}V   
\]
for all $d\geq 2$.  We also have
\[
(S^d V)^*\cong S^dV^*\cong S^d\Pi V\cong\Pi^d\Lambda^dV
\]
and hence we have the element $c\in\Pi^2\Lambda^2V^*=(S^2V)^*$ which is adjoint to $q$, i.e. for $x\in S^dV$, $y\in\Pi^{d-2}\Lambda^{d-2}V$, we have
\[
(q(x),y)=(-1)^{|x|}(x,cy)
\]
where $(-,-)$ denotes the pairing of dual vector spaces.  Since $c$ is odd, $c^2=0$, so we get that $q^2=0$.  

As a $\g_{\ol{0}}$-module we have the decomposition 
\[
S^dV=\bigoplus\limits_{\makecell{i+j=d\\ j\leq n}}S^iL_0(\epsilon_1)\otimes\Lambda^jL_0(-\epsilon_n)
\]
By Pieri's rule, we get
\[
S^iL_0(\epsilon_1)\otimes\Lambda^jL_0(-\epsilon_n)=L_{0}(i\epsilon_1-\epsilon_{n-j+1}-\dots-\epsilon_n)\oplus L_{0}((i-1)\epsilon_1-\epsilon_{n-j+2}-\dots-\epsilon_n)
\]
when $1\leq j<n$, and if $j=n$ we only get the first factor.  

For $j\neq n$, as a $\g_{\ol{0}}$-module the kernel of $q$ on $S^iV_{\ol{0}}\otimes\Lambda^jV_{\ol{1}}^*$ will be
\[
L_{0}(i\epsilon_1-\epsilon_{n-j+1}-\dots-\epsilon_n)
\]
For $j=n$, the kernel of $q$ on $S^iL_0(\epsilon_1)\otimes\Lambda^jL_0(-\epsilon_n)$ will be everything if $d=n$ (equivalently $i=0$), or the kernel will be trivial if $d>n$ (equivalently $i>0$).  Hence, as a $\g_{\ol{0}}$-module, the kernel of $q$ on $S^dV$ will be
\[
\bigoplus\limits_{0\leq i<\min(d,n)} L_{0}((d-i)\epsilon_1-\epsilon_{n-i+1}-\dots-\epsilon_n).
\]
for $d\neq n$, and for $d=n$ we get
\[
\bigoplus\limits_{0\leq i\leq n} L_{0}((n-i)\epsilon_1-\epsilon_{n-i+1}-\dots-\epsilon_n).
\]
We can write $S^\bullet V=\C[x_1,\dots,x_n,\xi_1,\dots,\xi_n]$ where $x_i$ has weight $\epsilon_i$, $\xi_i$ weight $-\epsilon_i$.  Then $x_1$ is our highest weight vector, and so $x_1^d$ will be a highest weight vector for all $d\geq 0$.  As a $\g_{\ol{0}}$-module, $x_1^d$ generates $S^dV_{\ol{0}}$.  In $\g_{\ol{1}}$ we have all differential operators $\xi_i\d_{x_j}-\xi_j\d_{x_i}$, for $i<j$.  Applying to $x_1^d$, sequentially, 
\[
\xi_n\d_{x_1}-\xi_1\d_{x_n},\ \ \xi_{n-1}\d_{x_1}-\xi_1\d_{x_{n-1}},\ \ \dots,\ \ \xi_{2}\d_{x_1}-\xi_1\d_{x_2}
\]
we get, up to scalar,
\[
\xi_nx_1^{d-1},\ \ \xi_{n-1}\xi_nx_1^{d-2},\ \ \dots,\ \ \xi_2\cdots\xi_nx_1^{d-n+1}
\]
these are exactly the generators of the irreducible $\g_{\ol{0}}$-summands of the kernel of $q$ for $d\neq n$.  For $d=n$, we also get the one-dimensional $\g_{\ol{0}}$-module generated by
\[
\xi_1\cdots\xi_n.
\]
However, this element is not in $\UU\g\cdot x_1^n$; indeed, $x_1^n$ is in the $\g$-submodule given by the image of $q:S^{n+2}V\to S^nV$, but $\xi_1\cdots\xi_n$ is not in the image of $q$.  Further, by applying operators $x_i\d_{\xi_i}$ from the $\p(n)$-action we can get back to $S^dV_{\ol{0}}$ from any of the above elements $\xi_{i}\cdots\xi_n x_1^{d-n+i}$.  It follows from the above analysis that $\ker q\cong L_{\b}(d\epsilon_1)$ for $d\neq n$.  If $d=n$, $\ker q$ is a nontrivial extension of $L_{\b}(-\omega)$ by $L_{\b}(n\epsilon_1)$.  

Now we claim that $S^0(V)=\C$, $S^1V=V$, and for $d\geq 2$ we have the following socle filtrations:

\renewcommand{\arraystretch}{1.4}
\[
S^dV=\begin{tabular}{|c|}
\hline 
$\Pi L((d-2)\epsilon_1)$\\
\hline
$L(d\epsilon_1)$\\
\hline
\end{tabular},
\]
unless $d=n$, in which case we get
\[
S^nV=\begin{tabular}{|c|}
\hline 
$\Pi L((n-2)\epsilon_1)$\\
\hline
$\Pi^nL(-\omega)$\\
\hline
$L(n\epsilon_1)$\\
\hline
\end{tabular}.
\]
\renewcommand{\arraystretch}{1}

Further, in each of the above cases, the radical of the module is equal to $\ker q$.

Proof of claim: we have already computed the kernel of $q$ in each case and have shown it is indeed the radical of each module shown above.  Further, our analysis shows that as a $\g_{\ol{0}}$-module, we get an isomorphism (for $d\geq 2$)
\[
q:S^dV/\ker q\to L_{\b}((d-2)\epsilon_1)
\]
It follows that this must be an isomorphism of $\g$-modules.  To finish the proof, we need to show the socle filtration is as advertised.  We see that $q(x_1\xi_1\cdots\xi_{n-1})\neq$, and further
\[
(\xi_n\d_{x_1}-\xi_1\d_{x_n})(x_1\xi_1\cdots\xi_{n-1})=(-1)^{n-1}\xi_1\cdots\xi_n.
\]
This completes the proof.

\subsection{Symmetric algebras for $\q(n)$-modules}

By Schur-Sergeev duality (see \cite{cheng2012dualities} chapter 3), we have that $S^m(Q_{n|n}^*)$ is irreducible of highest weight $-m\epsilon_n$ with respect to the standard Borel for all $m,n$.

For the family of typical spherical modules for $\q(2)$, we compute the symmetric algebras of the dual with respect to the Borel $(\b^{st})^{op}$.  This leads to a canonical identification $L_{\b^{st}}(\lambda)^*\cong L_{(\b^{st})^{op}}(-\lambda)$ when the length of $\lambda$ is $2$.  We then have for $t\neq -1/2$, by the character formula for typical $\q(2)$-modules, 
\[
S^d(L((t+1)\epsilon_1+t\epsilon_2)^*)\cong \Pi^dL_{(\b^{st})^{op}}(-d(t+1)\epsilon_1-dt\epsilon_2)
\]
and
\[
S^d(\Pi L((t+1)\epsilon_1+t\epsilon_2)^*)\cong L_{(\b^{st})^{op}}(-d(t+1)\epsilon_1-dt\epsilon_2)
\]

\section{Appendix A: Root Systems and Borels for $\g\l(m|n)$ and $\o\s\p(m|2n)$}\label{App_borels}

Here we give a description of the root system and weight space notation for $\g\l(m|n)$ and $\o\s\p(m|2n)$.  We also describe our notation for certain Borel subalgebras.

\

Let $\g$ be either $\g\l(m|n)$, $\o\s\p(2m|2n)$, or $\o\s\p(2m+1|2n)$ and let $\h_{\ol{0}}\sub\g$ be the Cartan subalgebra of diagonal matrices.  Write $(-,-)$ for the restriction of a fixed invariant form to $\h_{\ol{0}}$.  Then there is basis of $\h_{\ol{0}}^*$ given by $\epsilon_1,\dots,\epsilon_m,\delta_{1},\cdots,\delta_{n}$, where
\[
(\epsilon_i,\epsilon_j)=-(\delta_{i},\delta_j)=\delta_{ij},\ \ \  (\epsilon_i,\delta_j)=0.
\]
The root system for $\g\l(m|n)$ has even and odd components
\[
\Delta_{\ol{0}}=\{\epsilon_i-\epsilon_j,\delta_k-\delta_l\}_{i\neq j,k\neq l}, \ \ \ \Delta_{\ol{1}}=\{\pm(\epsilon_i-\delta_j)\},
\]
the root system for $\o\s\p(2m|2n)$ has even and odd components
\[
\Delta_{\ol{0}}=\{\pm\epsilon_i\pm\epsilon_j,\pm\delta_k\pm\delta_l,\pm2\delta_k\}_{i\neq j,k\neq l}, \ \ \ \Delta_{\ol{1}}=\{\pm\epsilon_i\pm\delta_j\},
\]
and the root system for $\o\s\p(2m+1|2n)$ has even and odd components
\[
\Delta_{\ol{0}}=\{\pm\epsilon_i\pm\epsilon_j,\pm\epsilon_i,\pm\delta_k\pm\delta_l,\pm2\delta_k\}_{i\neq j,k\neq l}, \ \ \ \Delta_{\ol{1}}=\{\pm\epsilon_i\pm\delta_j,\pm\delta_j\}.
\]
Let $W$ denote the Weyl group of the even part of each root system above.  Then the conjugacy classes of Borel subalgebras of $\g$ are in bijection with a choice of simple roots up to the $W$-action.  For each algebra, there is a well-known classification of such conjugacy classes in terms of $\epsilon\delta$-sequences (see for instance \cite{cheng2012dualities}, section 1.3).  When discussing which (conjugacy classes of) Borels a representation is spherical with respect to for one of these algebras, we will speak in terms of $\epsilon\delta$-sequences.

For example, for $\g\l(1|2)$ there are three conjugacy classes of Borels, corresponding to the sets of simple roots $\{\epsilon_1-\delta_1,\delta_1-\delta_2\}$, $\{\delta_1-\epsilon_1,\epsilon_1-\delta_2\}$, and $\{\delta_1-\delta_2,\delta_2-\epsilon_1\}$.  To each of these, in order, we associate the $\epsilon\delta$-sequence $\epsilon\delta\delta$, $\delta\epsilon\delta$, and $\delta\delta\epsilon$.  We may also occasionally write, e.g., $\delta^2\epsilon$ for the string $\delta\delta\epsilon$.

Now for each superalgebra $\g$ that we consider, we would like to define a map:
\[
\{\epsilon\delta\text{-sequences}\}\to \{\text{Borels }\b\sub\g\}, \ \ \ \sigma\mapsto \b^{\sigma}
\]
such that the conjugacy class of $\b^{\sigma}$ corresponds to the simple root system defined by $\sigma$.

Since we have already chosen a Cartan subalgebra, and an $\epsilon\delta$-sequence specifies a conjugacy class of Borel, to define $\b^{\sigma}$ it suffices to rigidify the $W$-symmetry by making a choice of even Borel subalgebra, or equivalently a choice of simple roots for the even root system.  We now do this for each algebra $\g$.

For $\g\l(m|n)$, we choose upper diagonal matrices, i.e. the even simple roots
\[
\{\epsilon_i-\epsilon_{i+1}\}_{1\leq i\leq m-1}\cup\{\delta_{j}-\delta_{j+1}\}_{1\leq j\leq n-1}.
\]
For $\o\s\p(2m|2n)$, we choose the even simple roots
\[
\{\epsilon_i-\epsilon_{i+1},\epsilon_{m-1}+\epsilon_m\}_{1\leq i\leq m-1}\cup\{\delta_{j}-\delta_{j+1},2\delta_{n}\}_{1\leq j\leq n-1}.
\]
And for $\o\s\p(2m+1|2n)$, we choose the even simple roots 
\[
\{\epsilon_i-\epsilon_{i+1},\epsilon_m\}_{1\leq i\leq m-1}\cup\{\delta_{j}-\delta_{j+1},2\delta_n\}_{1\leq j\leq n-1}
\]
For example, if $\g=\g\l(m|n)$, then $\b^{\epsilon\cdots\epsilon\delta\cdots\delta}=\b^{\epsilon^m\delta^n}$ has simple root system
\[
\{\epsilon_1-\epsilon_2,\dots,\epsilon_{m}-\delta_1,\delta_1-\delta_2,\cdots\delta_{n-1}-\delta_n\}.
\]
This is the subalgebra of upper-triangular matrices in $\g\l(m|n)$.  

Or, if $\g=\o\s\p(4|4)$, then $\b^{\epsilon(-\epsilon)\delta\delta}=\b^{\epsilon(-\epsilon)\delta^2}$ has simple root system
\[
\{\epsilon_1+\epsilon_2,-\epsilon_2-\delta_1,\delta_1-\delta_2,2\delta_2\}
\]
We now have made choices for representatives of each conjugacy class of Borel subalgebras of $\g$.

\bibliographystyle{amsalpha}
\bibliography{bibliography}

\textsc{\footnotesize Dept. of Mathematics, University of California at Berkeley, Berkeley, CA 94720} 

\textit{\footnotesize Email address:} \texttt{\footnotesize alexsherman@berkeley.edu}

\end{document}